\documentclass[11pt]{scrartcl}

\setkomafont{subsubsection}{ \rmfamily}
\setkomafont{subsection}{\large \rmfamily}
\setkomafont{section}{\Large \rmfamily}

\usepackage{scrlayer-scrpage}

\usepackage{lipsum}
\usepackage{mwe}
\usepackage{hyphenat}
\usepackage{typearea}
\usepackage{graphicx}
\graphicspath{ {images/} }
\usepackage[utf8]{inputenc}
\usepackage[T1]{fontenc}
\usepackage[notext]{stix}
\usepackage{step}
\usepackage[11pt]{moresize}

\usepackage{anyfontsize}
\usepackage{t1enc}
\usepackage[english]{babel}
\usepackage{amsmath}
\usepackage{mathtools}
\usepackage{amsfonts}
\usepackage{amsthm}
\usepackage{xcolor}
\usepackage{enumitem}
\usepackage{epsfig,epstopdf,titling,url,array,hyperref}
\usepackage{scalerel}
\usepackage{float}
\usepackage{tikz-cd}

\newtheorem{theorem}{Theorem}

\newtheorem{lemma}[theorem]{Lemma}
\newtheorem{proposition}[theorem]{Proposition}
\theoremstyle{definition}

\theoremstyle{remark}
\newtheorem{remark}{Remark}

\newcounter{example}[section]
\newenvironment{example}[1][]{\refstepcounter{example}\par\medskip
   \noindent \textbf{Example~\theexample. #1} \rmfamily}{ \medskip}

\begin{document}
\title{FAMILIES OF PAIRS OF GENUS $2$ CURVES WITH ISOMORPHIC UNPOLARIZED JACOBIANS}
\author{\text{RAGHDA ABDELLATIF}\thanks{Institut für Mathematik,  Humboldt-Universität zu Berlin, Unter den Linden 6, 10099 Berlin, Germany. }\\
\texttt{raghda.abdellatif@hu-berlin.de}
}
\date{ }
\maketitle
\begin{abstract}
We construct three families of pairs of genus $2$ curves over a field $K$, whose Jacobians are isomorphic as unpolarized abelian varieties. Each family is parameterized by an open subset of $\mathbb{P}^1_K$. Our construction is based on a remark in a paper by Howe \cite{Howe}. Later, we describe an obstruction to perusing the same technique over $\mathbb{Q}$.
\end{abstract}

\section{Introduction}
The Torelli says that every non-singular projective algebraic curve is uniquely determined by its principally polarized Jacobian. On the other hand, every principally polarized abelian surface is either the Jacobian of a smooth genus 2 curve or the canonically polarized product of two elliptic curves. The latter statement can, for example, be viewed as a direct consequence of the Matsusaka-Ran criterion \cite{BLComplex}. The multitude of principal polarizations on an abelian surface thus indicates that an unpolarized Jacobian of a curve only provides partial some geometric information of the curve. A pair of curves with isomorphic unpolarized Jacobians can be used to prove that a particular statement about a curve can not be deduced from its unpolarized Jacobian. For example it is known \cite{Howe} that whether a genus $3$ curve is hyperelliptic or a plane quartic cannot be deduced from its unpolarized Jacobian. A shift of focus to arithmetic properties of curves that cannot be read of their unpolarized Jacobian invokes a search for such pairs over number fields or finite fields. With that motivation in mind, explicit families over an arbitrary field $K$ of characteristic different from $2$ were constructed in \cite{Howe}. Each of these families is parameterized by an open subset of $\mathbb{P}^1_K$. In particular, over any number field each family consists of  infinitely many pairs. The aim of this work is to replicate the construction of the first family, as suggested by a remark in the same paper, in order to construct three more families. Roughly speaking, the idea is to start from an isogeny $\psi: E \rightarrow E'$ of elliptic curves over a field $K$. Then the principally polarized abelian variety $A:=E \times E'$ comes equipped with a non-trivial automorphism $\alpha: (P,P') \mapsto (P,P'+\psi(P))$. One then attempts to find a maximal subgroup $G$ of $A[2]$ that is isotropic with respect to the $e_2$-pairing on $A$. With the right conditions on $G$, the quotient $A/G$ is isomorphic to the Jacobian of a curve $C$, whose defining equation can be explicitly written. The curve can be defined over $K$, if the subgroup $G$ is invariant under the action of $\text{Gal}(\overline{K}/K)$. By the $5$-lemma, the quotients $A/G$ and $A/\alpha(G)$ are isomorphic unpolarized abelian varieties. Therefore, we have a chance at constructing a pair of curves with isomorphic unpolarized Jacobians, if the the quotient $A/\alpha(G)$ is isomorphic to the Jacobian of a curve $C'$.
Unfortunately, it is not guaranteed that the curves $C$ and $C'$ are non-isomorphic. For that reason, Igusa invariants were used to check whether the two curves are geometrically distinct. In \cite{Howe}, Howe considers a general family of isogenies, $\psi$, of elliptic curves of degree $2$. The main goal of this work is to exhaust all the possibilities for a similar construction when considering isogenies of higher degrees that can be defined over $\mathbb{Q}$. At first, a general family of degree $3$ isogenies of elliptic curves is considered. The requirement of Galois equivariant of $G$ imposes restrictions on the discriminants of elliptic curves involved; let's agree to refer to such restrictions as  \textit{Galois restrictions}. Unlike the first example in \cite{Howe}, the fact that the degree of $\psi$ is odd imposes a different Galois restriction on the curves' discriminants. In general, for isogenies of even degree, it is required that the discriminant of $E$ differs by a square, in $K$, from that of $E'$. However, for odd degrees, it is required that the discriminant of $E$, and thus also $E'$, is itself a square in $K$. That restriction emerges from the necessity of the Weierstrass points of $E$ to be all defined either over $K$, or over a Galois cubic extension of $K$. The first example, thus constructed, is summarized in the following theorem, whose proof is the subject of section \ref{example1}:
\begin{theorem}
Suppose $\text{char}(K) \neq 2$, and let $t \in K$ be such that 
\begin{equation*}
t(t^2+27)(t^2+243)(t^2+3)(t^4 - 10t^2 + 729)\neq 0.
\end{equation*}
Let $C_t$ be the genus $2$ curve defined by the equation
\begin{align*}
&(t^2+27)(t^2-8t+27)y^2=16t^3x^6+(t^4+16t^3-126t^2+648t-2187)x^4 +\\
&(t^4-8t^3+42t^2-144t-243)x^2-16t.
\end{align*}
The Jacobians of $C_t$ and $C_{-t}$ are isomorphic as unpolarized abelian varieties. Moreover, the curves $C_t$ and $C_{-t}$ are geometrically non-isomorphic unless $\text{char}(K)=13$ and $t^2 + 2t =-12$, or
 $\text{char}(K)=17$ and $t^2=-7$. 
\end{theorem}

The next step will be to consider a general family of degree $4$ isogenies of elliptic curves. The example thus constructed is a closer replica of the first example in \cite{Howe}. The following theorem, whose proof is detailed in section \ref{example2}, describes that second family.
\begin{theorem}
Suppose $\text{char}(K) \neq 2$, and let $t \in K$ be such that 
\begin{equation*}
(t^2+1)(2t^2+1)(t^2+2)(t^2-1)(t^4+t^2+1)\neq 0.
\end{equation*} Let $C_t$ be the genus $2$ curve defined by the equation
\begin{align*}
& (t^2+1)(t^2-t+1)(t-1)y^2= (4x^2+t)(16t^4x^4+8(2t^4 - 4t^3 + 5t^2 - 4t + 2)x^2+1).
\end{align*}
The Jacobians of $C_t$ and $C_{-t}$ are isomorphic as unpolarized abelian varieties. Moreover, the curves $C_t$ and $C_{-t}$ are geometrically non-isomorphic  unless char$(K)=23$ and $t^2 \in \{10,7 \}$, or char$(K)=47$ and $t^2 \in \{ 26,38\}$.
\end{theorem}
The third example resembles the construction in the first one, by considering a family of elliptic curves that admit a $K$-rational isogeny of degree $7$. The proof of the following theorem can be found in section \ref{example3}.
\begin{theorem}
Suppose $\text{char}(K) \neq 2$, and let $t \in K$ be such that 
\begin{align*}
&t(t^4+13t^2+49)(t^8 - 6t^6 + 43t^4 - 294t^2 + 2401)\\
&\cdot(t^4 + 5t^2 + 1)(t^2+7)(t^4 + 245t^2 + 2401)\neq 0.
\end{align*}
Let $C_t$ be the genus $2$ curve defined by the equation
\begin{align*}
& (t^4+5t^2+1)(t^2-5t+7)(t^2-3t+7)y^2= 16t^7x^6 + (t^8 + 16t^7 - 130t^6 + 640t^5 - 2289t^4 + \\
&6272t^3 - 13034t^2
    + 19208t - 16807)x^4 + (t^8 - 8t^7 + 38t^6 - 128t^5 + 327t^4 - 640t^3
    +\\
    & 910t^2 - 784t - 343)x^2 - 16t.
\end{align*}
The Jacobians of $C_t$ and $C_{-t}$ are isomorphic as unpolarized abelian varieties. Moreover, the curves $C_t$ and $C_{-t}$ are geometrically non-isomorphic unless $\text{char}(K)=13$ and $t^2  =7$, or
 $\text{char}(K)=17$ and $(t^2 + 3t + 10)(t^2 + 8 t + 1)(t^2 + 9t+ 10)(t^2 + 14t + 10)=0$, or
 $\text{char}(K)=41$ and $(t^2 + t + 34)(t^2 - t+ 34)=0$.
\end{theorem}
Unfortunately, this is the furthest one can go over $\mathbb{Q}$ with such constructions, as we will discuss in section \ref{s2o}. The problem is that the imposed Galois restrictions are far too strong over $\mathbb{Q}$ that at most finitely many elliptic curves can satisfy them. Further work can be done by constructing other automorphisms instead of $\alpha$, or when considering other number fields instead of $\mathbb{Q}$. \\

In a forthcoming work we computed the Automorphism groups, over $\mathbb{C}$, of all curves in each of the families. We found  examples of pairs of curves whose Jacobians are isomorphic, and yet they have different automorphism groups over $\mathbb{C}$. In most of the examples that we found, the previous statement still holds over any algebraically closed field of any positive characteristic different from $2$.

\subsection*{Acknowledgments.} I would like to thank Berlin Mathematical School for their financial support during the production of this work. I owe Angela Ortega my immense gratitude for her supervision, mentorship, and the many hours she devoted to discussions of several details of this work. I am also grateful to Thomas Krämer for the co-supervision, thorough revision of the work, and his helpful comments. Finally, I thank Marwan Benyoussef for reviewing and commenting on the script.

   \section{Rational Cyclic Isogenies of Elliptic Curves}
Let $X_0(N)$ be the modular curve of elliptic curves with a cyclic $K$-rational torsion subgroup or order $N$.
It is known that the genus of $X_0(N)$ is zero for $ N \in \{1,2,3,4,5,6,7,8,9,$ $10,12,13,16,18,$ $25 \}$. That allows for a rational parameterization of cyclic rational isogenies in those degrees. A systematic construction of parameterizations of pairs of elliptic curves $(E,E')$ over $\mathbb{C}$, with $E$ cyclically $N$-isogeneous to $E'$, can be found in \cite{Maier}. Those parameterizations are summarized in the following table that corresponds to table 7 in \cite{Maier}. Here $j_N$ and $j'_N$ refer to the $j$-invariant of $N$-isogeneous elliptic curves $E$ and $E'$, respectively. 

\begin{table}[H]\label{t1}
\centering
\caption{Parameterization of cyclic isogenies over $\mathbb{C}$, for all degrees $N$ for which the curve $X_0(N)$ is of genus zero \protect    }
\renewcommand{\arraystretch}{1.5} 
\begin{tabular}{|m{1em}|m{35.7em}|}    \hline
   $N$ & $j_N$ , $j'_N$\\
   \hline
    $2$ & $\frac{(s+16)^3}{s}$ , $\frac{(s+256)^3}{s^2}$\\
     \hline
    $3$ & $\frac{(s+27)(s+3)^3}{s}$ , $\frac{(s+27)(s+243)^3}{s^3}$\\
     \hline
      $4$ & $\frac{(s^2+16s+16)^3}{s(s+16)}$ , $\frac{(s^2+256s+4096)^3}{s^4(s+16)}$\\
     \hline
    $5$ & $\frac{(s^2+10s+5)^3}{s}$, $\frac{(s^2+250s+3125)^3}{s^5}$\\
    \hline
    $6$ & $\frac{(s+6)^3(s^3+18s^3+84s+24)^3}{s(s+8)^3(s+9)^2}$ , $\frac{(s+12)^3(s^3+252s^2+3888s+15552)^3}{s^6(s+8)^2(s+9)^3}$\\
     \hline
    $7$ & $\frac{(s^2+13s+49)(s^2+5s+1)^3}{s}$ , $\frac{(s^2+13s+49)(s^2+245s+2401)^3}{s^7}$\\
     \hline
    $8$ & $\frac{(s^4+16s^3+80s^2+128s+16)^3}{s(s+4)^2(s+8)}$ , $\frac{(s^4+256s^3+5120s^2+32768s+65536)^3}{s^8(s+4)(s+8)^2}$\\
     \hline
    $9$ & $\frac{(s+3)^3(s^3+9s^2+27s+3)^3}{s(s^2+9s+27)}$ , $\frac{(s+9)^3(s^3+243s^2+2187s+6561)^3}{s^9(s^2+9s+27)}$\\
     \hline
    $10$ & $\frac{(s^6+20s^5+160s^4+640s^3+1280s^2+1040s+80)^3}{s(s+4)^4(s+5)^2}$ , $\frac{(s^6+260s^5+6400s^4+64000s^3+320000s^2+800000s+800000)^3}{s^{10}(s+4)^2(s+5)^5}$\\
     \hline
    $12$ & $\frac{(s^2+6s+6)^3(s^6+18s^5+126s^4+432s^3+732s^2+504s+24)^3}{s(s+2)^3(s+3)^4(s+4)^3(s+6)}$ , \\
 &   $\frac{(s^2+12s+24)3(s^6+252s^5+4392s^4+31104s^3+108864s^2+186624s+124416)^3}{s^{12}(s+2)(s+3)^3(s+4)^4(s+6)3}$\\
     \hline
    $13$ & $\frac{(s^2+5s+13)(s^4+7s^3+20s^2+19s+1)^3}{s}$ , $\frac{(s^2+5s+13)(s^4+247s^3+3380s^2+15379s+28561)^3}{s^{13}}$\\
     \hline
    $16$ & $\frac{(s^8+16s^7+112s^6+448s^5+1104s^4+1664s^3+1408s^2+512s+16)^3}{s(s+2)^4(s+4)(s^2+4s+8)}$ ,\\
  &   $\frac{(s^8+256s^7+5632s^6+53248s^5+282624s^4+917504s^3+1835008s^2+2097152s+1048576)^3}{s^{16}(s+2)(s+4)^4(s2+4s+8)}$\\
     \hline
    $18$ & $\frac{(s^3+6s^2+12s+6)^3(s^9+18s^8+144s^7+666s^6+1944s^5+3672s^4+4404s^3+3096s^2+1008s+24)^3}{s(s+2)^9(s+3)^2(s^2+3s+3)^2(s^2+6s+12)}$ , \\
&    $\frac{(s^3+12s^2+36s+36)^3(s^9+252s^8+4644s^7+39636s^6+198288s^5+629856s^4+1294704s^3+1679616s^2+1259712s+419904)^3}{s^{18}(s+2)^2(s+3)^9(s^2+3s+3)(s^2+6s+12)^2}$\\
     \hline
    $25$ & $\frac{(s^{10}+10s^9+55s^8+200s^7+525s^6+1010s^5+1425s^4+1400s^3+875s^2+250s+5)^3}{s(s^4+5s^3+15s^2+25s+25)}$ , \\
  &  $\frac{(s^{10}+250s^9+4375s^8+35000s^7+178125s^6+631250s^5+1640625s^4+3125000s^3
    +4296875s^2+3906250s+1953125)^3}{s^{25}(s^4+5s^3+15s^2+25s+25)}$\\
    \hline
   \end{tabular}
   \end{table}
For each $N$, the roots of the denominator correspond to a cusp on $X_0(N)$, whose width is determined by the multiplicity of the root. On the other hand, the roots of each non-cubed factor of $j_N(s)$, along with the roots of each non-squared factors on $j_N(s)-12^3$ constitute the elliptic points on $X_0(N)$. Henceforth, we will only consider the smooth locus on $X_0(N)$, away from cusps and elliptic points.

\section{The basic Construction}\label{s21}
 
Constructing a pair of genus $2$ curves with isomorphic unpolarized Jacobians amounts to constructing a pair of indecomposable principal polarizations on an abelian surface: Let $(A,\lambda)$ be a principally polarized abelian surface over a field $K$. For a positive integer $n$, let $G$ is a subgroup of the $n$-torsion subgroup $A[n]$. Suppose that $G$ is maximally isotropic with respect to the $\lambda$-$e_n$-pairing on $A[n]$ associated to the polarization $\lambda$, see \cite{Mum}. Then the polarization $n\lambda$ on $A$, descends to a principal polarization $\mu$ on the quotient variety $B:=A/G$. Suppose further that $A$ has an automorphism $\alpha$ such that $G':=\alpha(G)$ is also a maximally isotropic subgroup with respect to the $\lambda$-$e_n$-pairing. Let $(B',\mu')$ be the principally polarized abelian surface obtained from $G'$ as described above. The abelian surfaces $B$ and $B'$ are isomorphic due to the following

\begin{proposition}[\cite{Howe}]\label{p4}
For a morphism $\delta$ of abelian variety, we denote its dual morphism by $\hat{\delta} $. The automorphism $\alpha$ of $A$ provides an isomorphism: $\beta_0:B \rightarrow B'$. Identify $B'$ with $B$ via the isomorphism $\beta_0$, then $\mu'=\hat{\beta}\mu\beta$, where $\beta$ is the image of $\alpha$ in $(\text{End}B)\otimes\mathbb{Q}$.
\end{proposition}
\begin{proof}
Consider the following diagram with exact rows and commutative solid arrows.
\[\begin{tikzcd}
	0 & G & A & B & 0 \\
	0 & {G'} & A & {B'} & 0 \\
	\\
	&& {}
	\arrow[from=1-1, to=1-2]
	\arrow[from=1-2, to=1-3]
	\arrow["\omega",from=1-3, to=1-4]
	\arrow[from=1-4, to=1-5]
	\arrow[from=2-1, to=2-2]
	\arrow[from=2-2, to=2-3]
	\arrow["\omega'",from=2-3, to=2-4]
	\arrow[from=2-4, to=2-5]
	\arrow["{\alpha|_{G}}"', from=1-2, to=2-2]
	\arrow["\alpha"', from=1-3, to=2-3]
	\arrow["{\beta_0}"', dashed, from=1-4, to=2-4]
\end{tikzcd}\] The isomorphism $\alpha$ induces an isomorphism $\beta_0$ by the $5$-lemma. If $B$ and $B'$ are identified via $\beta_0$, then $G'$ can be viewed as the kernel of the map $ \omega  \alpha^{-1}$. Denote the principal polarization induced on $B$ via the identification $\beta_0$ by, yet again, $\mu$.
On one hand, $n \lambda = \hat{\omega} \mu \omega$, and on the other hand, $n \lambda = \hat{\alpha}^{-1}\hat{\omega} \mu'\omega  \alpha^{-1}  $. In other words, $\mu'= \hat{\omega}^{-1}\hat{\alpha}\hat{\omega}\mu \omega \alpha \omega^{-1}$. Then $\beta:= \omega \alpha \omega^{-1}\in (\text{End}B) \otimes \mathbb{Q}$ has the stated properties. 
\end{proof}

The polarizations $\mu$, or $\mu'$, are not necessarily indecomposable. We call the group $G$ non-reducible if the induced polarization on the quotient is indecomposable. The questions of how maximal isotropic subgroups look like and when they are non-reducible was  answered in the literature, when $A$ is isomorphic to a product of elliptic curves, see \cite{Kuhn} and \cite{Kani}. Additionally, it is easy to construct a non-trivial automorphism $\alpha$ on an abelian surface, if it is isomorphic to a product of two isogeneous elliptic curves, see \cite{Howe}. From now on we take $n=2$, and Char$(K)\neq 2$.

\subsection{Weierstrass Equations}
%For simplicity, we try to apply the above construction when $n=2$, and $A$ is  isomorphic to a product of two cyclically isogeneous elliptic curves:
Fix a separably closed field $\overline{K}$. Suppose $C$ is a genus $2$ curve over $\overline{K}$ which is $(2,2)$-isogeneous to a product of two elliptic curves. In other words, there exist elliptic subcovers $f_E$ and  $f_{E'}$
\[\begin{tikzcd}
	{} & C \\
	E && E'
	\arrow["{f_E}"', from=1-2, to=2-1]
	\arrow["{f_{E'}}", from=1-2, to=2-3]
\end{tikzcd}\]
onto elliptic curves $E$ and $E'$ such that $\text{deg}(f_E)=\text{deg}(f_{E'})=2$. Let $A$ be the principally polarized abelian surface $E \times E'$ whose polarization $\lambda$ comes from the product of the canonical polarizations on $E$ and $E'$. The Weil pairings on $E[2]$ and $E'[2]$ combine to give an non-degenerate alternating pairing $e_2$ on $A[2]$. The induced map $J(C) \rightarrow A$ is surjective and its kernel is a maximally isotropic subgroup $G$ of $A[2]$. More importantly, $G$ is the graph of an isomorphism $E[2](\overline{K}) \rightarrow E'[2](\overline{K})$ of full $2$-torsion structures that is not a restriction of an isomorphism $E \rightarrow E'$ of elliptic curves, see \cite{Kuhn}. Consequently, the polarization $2\lambda$ descends to a principal polarization $\mu$ on $B:=A/G$, and such that $(B, \mu ) \cong (J(C), \nu)$, where $\nu$ is the canonical polarization on $J(C)$. For a variety $A$ over a field $K$, and a field extension $K < L$, denote by $A_L$, the base extension of $A$ to $L$. Keeping the same exact conditions on $G$, the converse also holds: 
\begin{proposition}[\cite{HoweT}, Proposition 3]
Let $E$ and $E'$ be elliptic curves over a field $K$, let $\overline{K}$ be a separable closure of $K$, let $A$ be the polarized abelian surface $E \times E'$, and let $G \subset A[2](\overline{K})$ be the graph of a group isomorphism $\varphi: E[2](\overline{K})\rightarrow E'[2](\overline{K})$. Then $G$ is a maximal isotropic subgroup of $A[2](\overline{K})$.  Furthermore, the quotient polarized abelian variety $A_{\overline{K}}/G$ is isomorphic to the polarized Jacobian of a curve $C$ over $\overline{K}$, unless $\varphi$ is the restriction to $E[2](\overline{K})$ of an isomorphism $E_ {\overline{K}} \rightarrow {E'}_{\overline{K}} $. If $\varphi$ gives rise to a curve $C$, then $C$ and the isomorphism $J(C) \cong A/G$ can be defined over $K$ if and only if $G$ can be defined over $K$, if and only if $\varphi$ is an isomorphism of Galois modules; that is, $\varphi$ commute with the action of the absolute Galois group $\text{Gal}(\overline{K},K)$.
\end{proposition}

Keep in mind that the condition that $\varphi$ is not a restriction of a geometric isomorphism of elliptic curves trivially holds, if the pair of curves are not geometrically isomorphic. A defining equation of the curve $C$ can be explicitly computed by means of the defining equations of $E$ and $E'$:

\begin{proposition}[\cite{HoweT}, Proposition 4]\label{P10}
Using the same notation as in the previous proposition, let $E: y^2=f$ and $E': y^2=g$ be elliptic curves, where $f,g \in K[x]$ are separable monic cubic polynomials with discriminants $\Delta_f, \Delta_g$, respectively. Suppose 
\begin{equation}
\varphi: E(\bar{K})[2] \rightarrow E'(\bar{K})[2]
\end{equation}
is a Galois-invariant isomorphism of groups that is not a restriction of an isomorphism of elliptic curves $E_{\overline{K}} \rightarrow {E'}_{\overline{K}} $.
For $i \in \{1,2,3 \}$, let $\alpha_i$, respectively $\beta_i$, be the roots of $f$, respectively $g$. Moreover, suppose $\varphi((\alpha_i,0))=(\beta_i,0)$. Let $a_1,a_2,b_1,b_2,A,B$ be defined as
\begin{align*}
a_1&:=(\alpha_3-\alpha_2)^2/(\beta_3-\beta_2)+(\alpha_2-\alpha_1)^2/(\beta_2-\beta_1)+(\alpha_1-\alpha_3)^2/(\beta_1-\beta_3),\\
a_2&:=\alpha_1(\beta_3-\beta_2)+\alpha_2(\beta_1-\beta_3)+\alpha_3(\beta_2-\beta_1),\\
b_1&:=(\beta_3-\beta_2)^2/(\alpha_3-\alpha_2)+(\beta_2-\beta_1)^2/(\alpha_2-\alpha_1)+(\beta_1-\beta_3)^2/(\alpha_1-\alpha_3),\\
b_2&:=\beta_1(\alpha_3-\alpha_2)+\beta_2(\alpha_1-\alpha_3)+\beta_3(\alpha_2-\alpha_1,\\
A&:=\Delta_g a_1/a_2,\\
B&:=\Delta_f b_1/b_2,
\end{align*} 
then $A , B \in K$, $a_1,a_2,b_1$ and $b_2$ are all nonzero. Moreover, the polynomial $h$ defined by
\begin{align*}
h:=&(A(\alpha_2-\alpha_1)(\alpha_1-\alpha_3)x^2+B(\beta_2-\beta_1)(\beta_1-\beta_3))\\
&\cdot(A(\alpha_3-\alpha_2)(\alpha_2-\alpha_1)x^2+B(\beta_3-\beta_2)(\beta_2-\beta_1))\\
&\cdot(A(\alpha_1-\alpha_3)(\alpha_3-\alpha_2)x^2+B(\beta_1-\beta_3)(\beta_3-\beta_2))
\end{align*}
is a separable sextic in $K[x]$, and the polarized Jacobian of $C: y^2=h$ is isomorphic to the quotient of $E \times E'$ by $G$, where $G$ is the graph of $\varphi$.
\end{proposition}

For $i,j \in \{ 1,2,3\} $, with $i \neq j$ set
\begin{align*}
\gamma_{ij}&:=(\beta_i-\beta_j)/(\alpha_i-\alpha_j),\\
\alpha&:=(a_3-a_2)(a_2-a_1)(a_1-a_3),\\
\beta&:=(b_3-b_2)(b_2-b_1)(b_1-b_3).
\end{align*}
We, then, get
\begin{equation*}
h=\kappa(\gamma_{32}x^2-1)\cdot(\gamma_{21}x^2-1)\cdot (\gamma_{13}x^2-1); \text{ where } \kappa:=A^3 \alpha^3/ \beta.
\end{equation*}

Let $\psi: E \rightarrow E'$ be an isogeny of elliptic curves: Then the product variety $E \times E'$ is equipped by a non-trivial automorphism
\begin{align*}
\alpha: (X,Y) \mapsto (X,Y+\psi(X)).
\end{align*}

Notice that the condition that $\alpha(G)$ is maximally isotropic, with respect to the $\lambda$-$e_n$-pairing, is not superfluous. That is, it could happen that $G$ is maximally isotropic while its image $\alpha(G)$ is not:
\begin{example}
Let $\psi$ be an isogeny of odd degree, and such that $E$ and $E'$ are not geometrically isomorphic. Let $\{P_i$, for $i \in \{ 1,2,3\} \}$, be the non-zero $2$-torsion points on $E$. Let $Q_i(:=\psi(P_i))$, for $i \in \{ 1,2,3\} $, be the non-zero $2$-torsion points on $E'$. Let $O_E$ and $O_{E'}$ be the zero point on $E$ and $E'$, respectively. Consider the automorphism $\alpha: E \times E' \rightarrow E \times E'$, $(P,Q)\mapsto (P,Q+\psi(P)) $. The graph 
\begin{equation*}
G:\{(O_E,O_{E'}),(P_1,Q_1),(P_2,Q_2),(P_3,Q_3) \}
\end{equation*} 
is a maximally isotropic subgroup of the $2$-torsion subgroup of $E \times E'$. Nonetheless, 
\begin{equation*}
\alpha(G)=\{(O_E,O_{E'}),(P_1,O_{E'}),(P_2,O_{E'}),(P_3,O_{E'}) \}
\end{equation*}  
is not a graph of an isomorphism of $2$-torsion subgroups, and therefore isn't isotropic with respect to the $\lambda$-$e_2$-pairing on $E \times E'$.
\end{example}\\

Remember that the goal was to produce two principal polarizations on an abelian surface. The above construction clearly does this, but the two polarizations are not necessarily distinct. According to proposition \ref{p4}, the two polarizations are distinct if and only if $\beta \notin \text{End}B$. It could happen, nonetheless, that $\beta \in \text{End}B$:

\begin{example} \label{ae2}
Let $E:y^2=f(x)$ be an elliptic curve over a field $K$. Assume that the cubic separable  polynomial $f$ splits over a cubic extension of $K$.  Let $A=E \times E$. Consider the map $\alpha: A \rightarrow A$, $(P,Q) \mapsto (P, Q +[m] Q)$, where $[m]$ is the multiplication  by $m$ map on $E$, for some odd integer $m$. Denote the three non zero $2$-torsion points of $E$ by $P_i$, for $i \in \{ 1, 2, 3 \}$. The two subgroups 
\begin{align*}
G_1&:=\{ (O,O),(P_1,P_2),(P_2,P_3),(P_3,P_1)\},\\
G_2&:=\{ (O,O),(P_1,P_3),(P_2,P_1),(P_3,P_2)\}.
\end{align*}
of $A[2]$ are maximally isotropic and are such that $\alpha (G_1)=G_2$. Generically, there exist two hyperelliptic curves $C_1$ and $C_2$ with $A/G_i \cong J(C_i)$, for $i \in \{1,2\}$. Nonetheless, it will always be the case that $C_1 \cong C_2$: Using the same notation as in the proof of proposition 4 in \cite{HoweT}, we denote the elliptic subcovers as

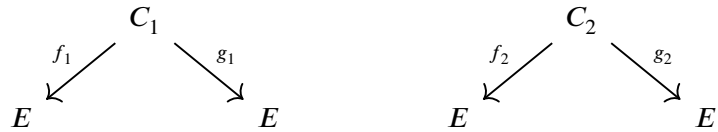
\[\begin{tikzcd}
	& {C_1} &&&& {C_2} \\
	E && E && E && E
	\arrow["{f_1}"', from=1-2, to=2-1]
	\arrow["{g_1}", from=1-2, to=2-3]
	\arrow["{f_2}"', from=1-6, to=2-5]
	\arrow["{g_2}", from=1-6, to=2-7]
\end{tikzcd}\]
Then one can show that the map \begin{equation*}
\alpha: C_1 \rightarrow C_2,\ \  (u,v) \mapsto (1/u,-v/u^3)
\end{equation*} is an isomorphism such $f_1= g_2 \circ \alpha$ and $g_1 =f_2 \circ \alpha$.

\end{example}

\subsection{Galois Restrictions}
We now investigate, when the group $G$ and its image $\alpha (G)$ are equivariant under the action of the absolute Galois group. Denote by $H$, the kernel of a cyclic isogeny $\psi: E \rightarrow E'$ of elliptic curves. We first study the case when the order of $H$ is even, and later when it is odd.
\subsubsection*{Even Degree}
Using the same notation as above, assume that the order of $H$ is even. Denote by $Q$, the unique $2$-torsion point in $H$. Let the other two nonzero $2$-torsion points in $E[2]$ be named $P$ and $R$. Moreover, denote by $Q'$ the image of $P$ and $R$ under the quotient by $H$ map. Moreover, denote by $P'$ and $R'$  the other two nonzero $2$-torsion points in $E'[2]$. In this setting, the points $Q$ and $Q'$ are $K$-rational on their respective curves. Assume that $E$ has exactly one nonzero $K$-rational $2$-torsion point. Then any Galois equivariant isomorphism $E[2] \xrightarrow{\cong} E'[2]$ must send $Q$ to $Q'$. Potentially, there are exactly two different isomorphisms,
\begin{equation*}
\varphi_i: E[2] \rightarrow E'[2], \  i \in \{ 1, 2 \},
\end{equation*}
which could be invariant under the action of the absolute Galois group. Their graphs of which would be 
\begin{align*}
G_{\varphi_1}:=&\{(O,O),(Q,Q'),(P,P'),(R,R')\},\\
G_{\varphi_2}:=&\{(O,O),(Q,Q'),(P,R'),(R,P') \}.
\end{align*}
Luckily, the automorphism $\alpha$ send the above graphs to each other. The only obstruction would be that the points $P$ on $E$ and $P'$ on $E'$, need not be defined over the same quadratic extension of $K$. To overcome this obstruction, it is necessary that the discriminants of the isogenous elliptic curves, $E$ and $E'$, differ by a square in $K$. We call this condition \textit{the Galois restriction in even degree.}

If all $2$-torsion points on $E$ are $K$-rational, it still suffices to consider only the above graphs: In case all $2$-torsion points on $E$ are rational there are $6$ different Galois invariant graphs. Those come in pairs each consisting of a graph and its image under the automorphism $\alpha$ as described above. However, the pair of curves obtained from each pair of graphs are just twists of the ones already described above.

\subsubsection*{Odd Degree}

Assume now that the order of $H$ is odd. For $j \in \{1, 2, 3 \}$, denote by $P_j$ the three distinct nonzero geometeric $2$-torsion points in $E[2]$. Moreover, denote by $P'_j$ the image of $P_j$ under the quotient by $H$ map in $E'[2]$. Let $G_{\varphi_i}$ be the graph of $\varphi_i$ for $i \in \{1 , 2\}$. Thinking ahead, the graph  $G_{\varphi_1}$ shouldn't have a point like $(P_j,P'_j)$ for some $j \in \{1,2,3 \}$; since $\alpha(P_j,P'_j)=(P_j,O_{E'})$, which can not be a point in a graph of any isomorphism. Thereupon, there are only two possible graphs that satisfy this criterion, namely,
\begin{align*}
G_{\varphi_1}&:=\{(O,O),(P_1,P'_2),(P_2,P'_3),(P_3,P'_1) \},\\
G_{\varphi_2}&:=\{O,O),(P_1,P'_3),(P_2,P'_1),(P_3,P'_2) \}.
\end{align*}
Fortunately, the automorphism $\alpha$ maps one to the other. \\

Since quotient map is assumed to be defined over $K$, for any field extension $L/K$ the point $P'_j$ is defined over $L$, whenever $P_j$ is. That fact reflects itself on the discriminant of $E$ being a square in $K$ multiple of that of $E'$.  It is not hard to see that each graph is Galois invariant if and only if the points $P_j$ and $P'_j$ are all rational over a Galois cubic extension of $K$. The last condition is equivalent to the discriminant of $E$, hence also that of $E'$, being a square in $K$:

\begin{lemma}\label{lgr}

Let $E:y^2=f(x)$ be elliptic curves, where $f \in K[x]$ is a separable monic cubic polynomial. $f(x)$ splits over a cubic Galois extension of $K$ if and only if the discriminant of the elliptic curve $E$ is a square.
\end{lemma}
\begin{proof}
If $f$ splits over $K$, then the discriminant of $E$ is a square. Otherwise, assume that $f$ is irreducible over $K$. The result follows, since the discriminant of the cubic extension $L:=K[x]/(f(x))$ is a square, if and only if, $L$ is a cubic Galois extension of $K$. Notice that the discriminant of $E$ is a square multiple the discriminant of $L$. 
\end{proof}
\textit{The Galois restriction in odd degree}  can then be stated as the condition that the discriminant of $E$ is a square in $K$.

\begin{remark}\label{geo} 
A pair of elliptic curves defined over $K$, are geometrically isomorphic, if and only if, they are isomorphic over, at most, a quadratic extension of $K$. Consequently, their respective discriminants vary by a square in $K$.
\end{remark}

We performed most of the following computations using the computer algebra Magama \cite{Mag}. Some of the routines were add to the repository:\\
\texttt{https://github.com/RaghdaAbdellatif/Pairs-Isomorphic-Jacobians}

\section{Howe's Example}\label{eH}
The original construction of the following example can be found in \cite{Howe}. Nonetheless, in this section, we briefly summarize this example: Fixing a parameter $s$ over a field $K$, we denote by
\begin{align*}
&(E_s, E'_s) \text{, a pair of $2$-isogeneous elliptic curves,}\\
&j_2(s), \text{ resp. }  j'_2(s) \text{, the j-invariant of $E_s$, resp. $E'_s$,}\\
&\Delta_2(s), \text{ resp. }  \Delta'_2(s)\text{, the discriminant of $E_s$, resp. $E'_s$.}\\
\end{align*}
We used Magma Computer Algebra \cite{Mag} to verify the isomorphism classes of the pair ($E_s$, $E'_s$), and compute their discriminants: \begin{align*}
j_2(s)=&\dfrac{(64s+16)^3}{64s},\\
E_s:y^2& =x( x^2 -4 (s + 1)x + 4(s + 1)),\\
\Delta_2(s)=&2^{12}s(s+1)^3,\\
j'_2(s)=&\dfrac{(64s+256)^3}{64s^2},\\
E'_s:y^2& = x(x^2 + 8(s + 1)x + 16s(s + 1)),\\
\Delta'_2(s)=&2^{18}s^2(s+1)^3.
\end{align*}
Owing to the Galois restrictions in the even degree, we require that $s=t^2$, for some parameter $t$ that takes value in $K$. The resulting family is described in the following theorem:
\begin{theorem}
Suppose $\text{char}(K) \neq 2$, and let $t \in K$ be such that 
\begin{equation*}
t(t^2-1)(t^2+1)\neq 0.
\end{equation*}
Let $C_t$ be the genus $2$ curve defined by the equation
\begin{equation*}
(t+1) y^2=(2x^2-t)(4t^2x^4+4(t^2+t+1)x^2+1)
\end{equation*}
The Jacobians of $C_t$ and $C_{-t}$ are isomorphic as unpolarized abelian varieties. Moreover, the curves $C_t$ and $C_{-t}$ are geometrically non-isomorphic unless $\text{char}(K)=11$ and $t^2 \in \{-3,-4\}$. 
\end{theorem}
\begin{remark}
In case $t=1$ or $-1$, the curves $E_t$ and $E'_t$ become isomorphic. Furthermore, the isomorphism $\varphi_1$, resp. $\varphi_2$, of $2$-torsion subgroups, in  case $t=-1$, resp. $t=1$ , is a restriction of an isomorphism over $K(i\sqrt{2})$ of elliptic curves. As a matter of course, the restriction $t^2-1\neq 0$ becomes essential.
\end{remark}

\begin{remark}
The two elliptic curves are geometrically isomorphic when, for example, $t^2-9/8t+1=0$. Nonetheless, for such $t$, the hyperelliptic curves $C_t$ and $C_{-t}$ are well defined and  are non-isomorphic. Indeed, the isomorphism $f:E_t \rightarrow E'_t, (x,y) \mapsto (1/5(8t - 12)x + 1/10(-21t + 24) , 1/5(8t + 8)y)$ sends $P \mapsto R', Q \mapsto P'$ and $R \mapsto Q'$. After identifying $E_t$ with $E'_t$ via the isomorphism $f$, the graphs $G_{\phi_1}$ and $G_{\phi_2}$ become
\begin{align*}
G_{\phi_1}&= \{(O,O),(P,Q),(Q,R),(R,P) \}\\
G_{\phi_2}&=\{ (O,O),(P,Q),(Q,P),(R,R)\},
\end{align*}
where neither is a restriction of an automorphism of $E_t$.
\end{remark}

\section{Proof of Theorem 1}\label{example1}
We use the above parameterization of the rational curve $X_0(3)$ to construct a one-parameter family of pairs of $3$-isogenous elliptic curves: We denote by
\begin{align*}
&(E_s, E'_s) \text{, a pair of $3$-isogeneous elliptic curves,}\\
&j_3(s), \text{ resp. } j'_3(s) \text{, the j-invariant of $E_s$, resp. $E'_s$,}\\
&\Delta_3(s), \text{ resp. } \Delta'_3(s)   \text{, the discriminant of $E_s$, resp. $E'_s$.}\\
\end{align*}
We used Magma Computer Algebra \cite{Mag}\footnote{In order to compute the isogeny $\psi: E_s \rightarrow E'_s$, we first used the function \textsf{DivisionPolynomial} to compute the $3$-division polynomial of $E_s$. We then extended the field of definition of $E_s$, to acquire a rational $3$-torsion point in order to feed it to the function  \textsf{IsogenyFromKernel}.} to construct a $3$-isogenous pair ($E_s$, $E'_s$) of elliptic curves, and compute their respective discriminants.

\begin{align*}
j_3(s)=&\dfrac{(s+27)(s+3)^3}{s},\\
E_s:y^2 &= x^3 + \frac{1}{4}(s^2 + 18s - 27)x^2 - 36sx -s^3 - 18s^2 + 27s,\\
\Delta_3(s)=&s(s+3)^6(s+27)^2,\\
j'_3(s)=&\dfrac{(s+27)(s+243)^3}{s^3},\\
E'_s:y^2= & x^3 +
\frac{1}{4}(s^2 + 18s - 27)x^2 - (5s^3 + 165s^2 + 891s + 1215)x \\
&-s^5 - 65s^4 - 1285s^3 - 7614s^2 - 17496s - 13851,\\
\Delta'_3(s)=&s^3(s+3)^6(s+27)^2.
\end{align*}

%\begin{align*}
%\psi: E &\rightarrow E' \\
%(x,y) \mapsto (&\dfrac{x^3 + v_1x^2 + v_2x +v_3}{x^2 + v_4x + v_5}   ,\dfrac{x^3y +
 %   w_1x^2y + w_2xy + w_3y}{x^3 + w_4x^2 + w_5 x + w_6});\\
  %  v_1&:=-6s^2 - 324s - 4374\\
   % v_2&:=9s^4 + 2268s^3 +
   % 144342s^2 + 3542940s + 30292137\\
   % v_3&:=- 3888s^5 - 478224s^4 - 23304672s^3
   % - 561201696s^2 - 6657892848s - 30993639120\\
   % v_4&:=-6s^2 - 324s -
   % 4374\\
   % v_5&:=9s^4 + 972s^3 + 39366s^2 + 708588s + 4782969\\ 
   % w_1&:=-9s^2 - 486s - 6561\\
   % w_2&:=27s^4 + 1620s^3 + 13122s^2 - 708588s -
  %  11160261\\
   % w_3&:=-27s^6 - 486s^5 + 136323s^4 + 7637004s^3 +
   % 141894747s^2 + 660049722s - 4261625379\\
   % w_4&:=-9s^2 - 486s -
    % 6561\\
    % w_5&:=27s^4 + 2916s^3 + 118098s^2 + 2125764s + 14348907\\  
    % w_6&:=-
 %   27s^6 - 4374s^5 - 295245s^4 - 10628820s^3 %- 215233605s^2 - 2324522934s
 %   - 10460353203
%\end{align*} 

The Galois restrictions in the odd degree imposes the condition $s=t^2$, for some $t \in K$. For $i \in \{1, 2, 3 \}$, let $(r_i,0)$ be the three distinct nonzero $2$-torsion points on $E_s$, and let $(s_i,0):=\psi((r_i,0))$ be their images on $E'_s$. Using the same notation as in proposition \ref{P10}, and the exposition afterwards. We put
\begin{equation*}
\tilde{\alpha}_i=\alpha_i=r_i,\ \ \ \tilde{\beta}_i=\beta_{\sigma(i)}=s_{\sigma^2(i)}.
\end{equation*}
where $\sigma$ is the cycle $(123)$ acting on $\{1, 2, 3 \}$. Here, we actually compute the quotient by the groups:
\begin{align*}
G_{\varphi_1}&:=\{(O,O),(P_1,P'_2),(P_2,P'_3),(P_3,P'_1) \}\\
G_{\varphi_2}&:=\{O,O),(P_1,P'_3),(P_2,P'_1),(P_3,P'_2) \}
\end{align*}
where $P_i:=(r_i,0)$ and $P'_i:=(s_i,0)$. We used Magma \cite{Mag} to compute the following:
\begin{align*}
\gamma_{21}\gamma_{13}\gamma_{32}&=\tilde{\gamma}_{21}\tilde{\gamma}_{13}\tilde{\gamma}_{32}=t^2\\
\gamma_{21}+\gamma_{13}+\gamma_{32}&=\dfrac{1}{16t} (t^4-8t^3+42t^2-144t-243), \\
\tilde{\gamma}_{21}+\tilde{\gamma}_{13}+\tilde{\gamma}_{32}&=-\dfrac{1}{16t}(t^4+8t^3+42t^2+144t-243) , \\
\gamma_{21}\gamma_{13}+\gamma_{32}\gamma_{21}+\gamma_{21}\gamma_{13}&=-\dfrac{1}{16t}(t^4+16t^3-126t^2+648t-2187) ,\\
\tilde{\gamma}_{21}\tilde{\gamma}_{13}+\tilde{\gamma}_{32}\tilde{\gamma}_{21}+\tilde{\gamma}_{21}\tilde{\gamma}_{13}&=\dfrac{1}{16t}(t^4-16t^3-126t^2-648t-2187) ,\\
\kappa=2^{-10}(t^2+3)^{21}(t&^2+27)^{8}t^9(t^2-8t+27)^3 ,\\
\tilde{\kappa}=- 2^{-10} (t^2+3)^{21}&(t^2+27)^{8}t^9(t^2+8t+27)^3,
\end{align*}
%\begin{align*}
%a_2&=81(t^2-8t+27)(t^2+27)^3 ,\\
%\tilde{a_2}&=-81(t^2+8t+27)(t^2+27)^3    ,
%\end{align*}
\begin{align*}
h_t&=2^{-10}(t^2+3)^{21}(t^2+27)^{8}t^8(t^2-8t+27)^3 \\
&\cdot(16t^3x^6+(t^4+16t^3-126t^2+648t-2187)x^4+(t^4-8t^3+42t^2-144t-243)x^2-16t), \\
\tilde{h}_t&=2^{-10}(t^2+3)^{21}(t^2+27)^{8}t^8(t^2+8t+27)^3\\
& \cdot(-16t^3x^6+(t^4-16t^3-126t^2-648t-2187)x^4+(t^4+8t^3+42t^2+144t-243)x^2+16t).
\end{align*}
By proposition $\ref{P10}$, the genus $2$ curves $y^2=h_t(x)$ and $y^2=\tilde{h}_t(x)$ have isomorphic unpolarized Jacobians insofar as $t(t^2+27)(t^2+243)(t^2+3)(t^4 - 10t^2 + 729)\neq 0$. Let $C_t$ be the genus $2$ curve defined by the equation
\begin{align*}
C_t:&(t^2+27)(t^2-8t+27)y^2= (16t^3x^6+(t^4+16t^3-126t^2+648t-2187)x^4+\\
&(t^4-8t^3+42t^2-144t-243)x^2-16t).
\end{align*}
The curve $y^2=h_t(x)$, resp. $y^2=\tilde{h}_t(x)$, is isomorphic over $K$ to the curve $C_t$, resp. $C_{-t}$ insofar as $t(t^2+27)(t^2+243)(t^2+3)(t^4 - 10t^2 + 729)\neq 0$.\\

In order to find the possible pathological cases in some positive characteristics, one computes the Igusa invariants, for example using the function \textsf{ScaledIgusaInvariants} in the computer algebra package Magma \cite{Mag}: Let $J_2(t), J_4(t),J_6(t),J_8(t), J_{10}(t) \in \mathbb{Z}(t) $ be  Igusa invariants of the twist $y^2=16t^3x^6+(t^4+16t^3-126t^2+648t-2187)x^4+(t^4-8t^3+42t^2-144t-243)x^2-16t$. The polynomials
\begin{align*}
R_2(t)&:=\dfrac{J_4(t)J_2(-t)^2-J_4(-t)J_2(t)^2}{t(t^2+27)^3(t^2+243)(t^2+3)},\\
R_3(t)&:=\dfrac{J_6(t)J_2(-t)^3-J_6(-t)J_2(t)^3}{t^5(t^2+27)^3(t^2+243)(t^2+3)},\\
R_5(t)&:=\dfrac{J_{10}(t)J_2(-t)^5-J_{10}(-t)J_2(t)^5}{t^5(t^2+27)^5(t^2+243)(t^2+3)}.
\end{align*}
are defined over $\mathbb{Z}[t]$. Using Magma, it was found that $ \text{gcd}(\text{resultant}(R_2,R_3),\text{resultant}(R_2,R_5))$
is divisible exactly by $2,3,5,13$ and $17$. So unless $\text{char}(K)$ is one of those values, the curves $C_t$ and $C_{-t}$ are not isomorphic for every $t\in K$ such that $t(t^2+27)(t^2+243)(t^2+3)(t^4 - 10t^2 + 729)\neq 0$.\\
\textit{Case 1:} If $\text{char}(k)=3$, the polynomials
\begin{align*}
R_2(t)&:=\dfrac{J_4(t)J_2(-t)^2-J_4(-t)J_2(t)^2}{t^{23}(t^2-1)^2},\\
R_3(t)&:=\dfrac{J_6(t)J_2(-t)^3-J_6(-t)J_2(t)^3}{t^{27}},\\
R_5(t)&:=\dfrac{J_{10}(t)J_2(-t)^5-J_{10}(-t)J_2(t)^5}{t^{55}(t^2-1)^2},
\end{align*}
are defined over  $\mathbb{F}_3[t]$. Moreover, $\text{gcd}(R_2,R_3)=\text{gcd}(R_2,R_5)=1$. Thus, the curves $C_t$ and $C_{-t}$ are not isomorphic for every $t\in K$ such that $t(t^2+27)(t^2+243)(t^2+3)(t^4 - 10t^2 + 729)\neq 0$.\\
\textit{Case 2:} If $\text{char}(k)=5$, the polynomials
\begin{align*}
R_2(t)&:=\dfrac{J_4(t)J_2(-t)^2-J_4(-t)J_2(t)^2}{t(t^2+27)^3(t^2+243)(t^2+3)^5},\\
R_3(t)&:=\dfrac{J_6(t)J_2(-t)^3-J_6(-t)J_2(t)^3}{t^5(t^2+27)^3(t^2+243)(t^2+3)^7},\\
R_5(t)&:=\dfrac{J_{10}(t)J_2(-t)^5-J_{10}(-t)J_2(t)^5}{t^5(t^2+27)^5(t^2+243)(t^2+3)^9},
\end{align*}
are defined over  $\mathbb{F}_5[t]$. Moreover, $\text{gcd}(R_2,R_3)=\text{gcd}(R_2,R_5)=1$. Thus, the curves $C_t$ and $C_{-t}$ are not isomorphic for every $t\in K$ such that $t(t^2+27)(t^2+243)(t^2+3)(t^4 - 10t^2 + 729)\neq 0$.\\
\textit{Case 3:} Assume $\text{char}(k)=13$. Define $R_2, R_3$ and $R_5$ as in the general case. On one hand,
\begin{align*}
\text{gcd}(R_2,R_3)=\text{gcd}(R_2,R_5)=(t^4+7t^2+1).
\end{align*}
On the other hand, if $t_0$ is such that $t_0^4+7t_0^2+1=0$, then $C(t_0)$ is isomorphic to 
\begin{equation*}
C: y^2=x^6 + 11x^5 + 7x^4 + 7x^2 + 2x + 1.
\end{equation*}
Thus, the curves $C_t$ and $C_{-t}$ are not isomorphic for every $t\in K$ such that $t(t^2+27)(t^2+243)(t^2+3)(t^4 - 10t^2 + 729)(t^4+7t^2+1)\neq 0$.\\
\textit{Case 4:} Assume $\text{char}(k)=17$. Similar to the case above, one finds
\begin{align*}
\text{gcd}(R_2,R_3)=\text{gcd}(R_2,R_5)=(t^2+7).
\end{align*}
Moreover, if $t_0^2=-7$, then $C(t_0)$ is isomorphic to
\begin{align*}
C: y^2=x^6 + 13x^5 + 13x^4 + 13x^2 + 4x + 1.
\end{align*}
Thus, the curves $C_t$ and $C_{-t}$ are not isomorphic for every $t\in K$ such that $t(t^2+27)(t^2+243)(t^2+3)(t^4 - 10t^2 + 729)(t^2+7)\neq 0$.
\begin{remark}
Roots of the polynomial $t(t^2+27)(t^2-8t+27)$ coincide with zeros of $J_{10}(t)$. It is worth noting that the first, resp. second, factor correspond to a cusp, resp. elliptic point, on the modular curve $X_0(3)$. $E_t$ and $E'_t$ are isomorphic, if $t$ is a  root of $t^2 - 8t^2 + 27$. It is, therefore, plausible to anticipate that for such values, the graph, $G_{\phi_1}$, is a restriction of some isomorphism of elliptic curves, and similarly for $-t$ and $G_{\phi_2}$. However, roots of $(t^2+3)(t^2+243)$ correspond to isomorphic pairs of Curves. One possible explanation for such occurrences, and also for similar occurrences in positive characteristics, is the scenario described before, and exampled by, example \ref{ae2}.
\end{remark}

\section{Proof of Theorem 2}\label{example2}
As in the previous proofs, we denote by
\begin{align*}
&(E_s, E'_s) \text{, a pair of $4$-cyclically-isogeneous elliptic curves,}\\
&j_4(s), \text{ resp. } j'_4(s) \text{, the j-invariant of $E_s$, resp. $E'_s$,}\\
&\Delta_4(s), \text{ resp. } \Delta'_4(s)   \text{, the discriminant of $E_s$, resp. $E'_s$.}\\
\end{align*}
We used Magma Computer Algebra \cite{Mag} to construct a $4$-cyclically-isogenous pair ($E_s$, $E'_s$) elliptic curves, and compute their respective discriminants.

\begin{align*}
j_4(s)=&\dfrac{(s^2-16s+16)^3}{s(s-16)},\\
E_s:y^2 &=  x(x^2 + s(s - 8)x + 16s^2),\\   
\Delta_4(s)=&2^{12}s^7(s-16),\\
j'_4(s)=&\dfrac{(s^2 + 224s + 256)^3}{s(s-16)^4},\\
E'_s:y^2&=  x^3 +
s(s - 8)x^2 -16 s^2(5s + 4)x -64s^3 (s-1)(s+8), \\
\Delta'_4(s)=&2^{12}s^7(s-16)^4.
\end{align*}
Due to the Galois restrictions in the even degree, we require that $s-16$ is a square in $K$. Accordingly, put $s=16t^2+16$. For $i \in \{1, 2, 3 \}$, let $(r_i,0)$ be the three distinct nonzero $2$-torsion points on $E_s$, where $\psi((r_1,0))=O_{E'}$. Let $(s_i,0)$ be the nonzero $2$-torsion points on $E'_s$, enumerated such that $\psi((r_2,0))=\psi((r_3,0))=(s_1,0)$. Using the same notation as in proposition \ref{P10}, and the exposition afterwards. We put
\begin{equation*}
\tilde{\alpha}_i=\alpha_i=r_i,\ \ \ \tilde{\beta}_i=\beta_{\sigma(i)}=s_{\sigma(i)}.
\end{equation*}
where $\sigma$ is the transposition $(23)$ acting on $\{1, 2, 3 \}$. Here we compute the quotient by the groups:
\begin{align*}
G_{\varphi_1}&:=\{(O,O),(P_1,P'_1),(P_2,P'_2),(P_3,P'_3) \}\\
G_{\varphi_2}&:=\{O,O),(P_1,P'_1),(P_2,P'_3),(P_3,P'_2) \}
\end{align*}
where $P_i:=(r_i,0)$ and $P'_i:=(s_i,0)$. Let $w$ be a square root of $t^2+1$.  We used Magma \cite{Mag} to compute the following:
\begin{align*}
&\tilde{\alpha}_1=\alpha_1=0 , \ \ \  \ &\tilde{\beta}_1=\beta_1&==-128(2s^4 + 5s^2 + 3),\\
&\tilde{\alpha}_2=\alpha_2=-64(s^2+1)(2s^2-2sw +1)  , \ \ \ \ &\tilde{\beta}_3=\beta_2&=-128(s^2+1)(4w-1),\\
&\tilde{\alpha}_3=\alpha_3=-64(s^2+1)(2s^2+2sw +1) , \ \ \ &\tilde{\beta}_2=\beta_3&=128(s^2+1)(4w+1).\\
\end{align*}
%\begin{align*}
%\gamma_{21}&=-8(t-1)(t^2-t+1)w-4(2t^4 - 4t^3 + 5t^2 - 4t + 1), \\
% \gamma_{13}&=8(t-1)(t^2-t+1)w -4(2t^4 -4t^3 +4t^2 -4t +2), \\
% \gamma_{32}&= \dfrac{-4}{t},\\
%\tilde{\gamma}_{21}&=-8(t+1)(t^2+t+1)w-4(2t^4 %+ 4t^3 + 5t^2 + 4t + 2), \\
% \tilde{\gamma}_{13}&=8(t+1)(t^2+t+1)w-4(2t^4 + 4t^3 + 5t^2 + 4t + 2), \\
 % \tilde{\gamma}_{32}&= \dfrac{4}{t},
%\end{align*}
\begin{align*}
\gamma_{32}&= \dfrac{-4}{t}, \ \ \  \tilde{\gamma}_{32}= \dfrac{4}{t}, \ \ \  \tilde{\gamma}_{21}\tilde{\gamma}_{13}=\gamma_{21}\gamma_{13}=16t^4,\\
\gamma_{21}+\gamma_{13}&=-8(2t^4 - 4t^3 + 5t^2 - 4t + 2), \ \ \  \tilde{\gamma}_{21}+\tilde{\gamma}_{13}=-8(2t^4 + 4t^3 + 5t^2 + 4t + 2),\\
\kappa&=- 2^{172}t^{11}(t-1)^3(t^2+1)^{25}(t^2-t+1)^3,\\
\tilde{\kappa}&=-2^{172}t^{11}(t+1)^3(t^2+1)^{25}(t^2+t+1)^3,
\end{align*}
%\begin{align*}
%a_2&=-2^{16}(t-1)(t^2+1)^2(t^2 - t + 1)w, \\
%\tilde{a_2}&=-2^{16}(t+1)(t^2+1)^2(t^2 + t + 1)w,
%\end{align*}
\begin{align*}
h_t&= 2^{172}t^{10}(t-1)^3(t^2+1)^{25}(t^2-t+1)^3(4x^2+t)(16t^4x^4+8(2t^4 - 4t^3 + 5t^2 - 4t + 2)x^2+1),\\
\tilde{h}_t&=2^{172}t^{10}(-t-1)^3(t^2+1)^{25}(t^2+t+1)^3(4x^2-t)(16t^4x^4+8(2t^4 + 4t^3 + 5t^2 + 4t + 2)x^2+1).
\end{align*}

By proposition \ref{P10}, the hyperelliptic curves $y^2=h_t(x)$ and $y^2=\tilde{h}_t(x)$ have isomorphic unpolarized Jacobian as long as $t(t^2+1)(2t^2+1)(t^2+2)(t^2-1)(t^4+t^2+1)\neq 0$. Let $C_t$ be the genus $2$ hyperelliptic curve defined by the equation
\begin{align*}
C_t: (t^2+1)(t^2-t+1)(t-1)y^2=(4x^2+t)(16t^4x^4+8(2t^4 - 4t^3 + 5t^2 - 4t + 2)x^2+1).
\end{align*}
The curve $y^2=h_t(x)$, resp. $y^2=\tilde{h}_t(x)$, is isomorphic over $K$ to the curve $C_t$, resp. $C_{-t}$. As in the previous two examples, Igusa invariants can be used to investigate the geometric isomorphism of the curves $C_t$ and $C_{-t}$: Let $J_2(t), J_4(t),J_6(t),J_8(t), J_{10}(t) \in  \mathbb{Z}(t) $ be  Igusa invariants of the twist $y^2=(4x^2+t)(16t^4x^4+8(2t^4 - 4t^3 + 5t^2 - 4t + 2)x^2+1)$. The polynomials
\begin{align*}
R_2(t)&:=\dfrac{J_4(t)J_2(-t)^2-J_4(-t)J_2(t)^2}{t(t^2+1)(2t^2+1)(t^2+2)},\\
R_3(t)&:=\dfrac{J_6(t)J_2(-t)^3-J_6(-t)J_2(t)^3}{t^5(t^2+1)(2t^2+1)(t^2+2)},\\
R_5(t)&:=\dfrac{J_{10}(t)J_2(-t)^5-J_{10}(-t)J_2(t)^5}{t^5(t^2+1)^3(2t^2+1)(t^2+2)}.
\end{align*}
are defined over $\mathbb{Z}[t]$. Using Magma \cite{Mag} one finds that, $\text{gcd}(\text{resultant}(R_2,R_3),\text{resultant}(R_2,R_5))$ is divisible exactly by $2,3,5,7,11,23,37$ and $47$. So unless $\text{char}(K)$ is one of those, the curves $C_t$ and $C_{-t}$ are not isomorphic for every $t\in K$ such that $t(t^2+1)(2t^2+1)(t^2+2)(t^2-1)(t^4+t^2+1)\neq 0$.\\
\textit{Case 1:} If $\text{char}(K)=3$, the polynomials
\begin{align*}
R_2(t)&:=\dfrac{J_4(t)J_2(-t)^2-J_4(-t)J_2(t)^2}{t(t^2+1)^3(2t^2+1)^{13}(t^2+2)},\\
R_3(t)&:=\dfrac{J_6(t)J_2(-t)^3-J_6(-t)J_2(t)^3}{t^5(t^2+1)(2t^2+1)^{11}(t^2+2)},\\
R_5(t)&:=\dfrac{J_{10}(t)J_2(-t)^5-J_{10}(-t)J_2(t)^5}{t^5(t^2+1)^3(2t^2+1)^{41}(t^2+2)},
\end{align*}
are defined over $\mathbb{F}_3[t]$. Moreover, $\text{gcd}(R_2,R_3)=\text{gcd}(R_2,R_5)=1$. Thus, the curves $C_t$ and $C_{-t}$ are not isomorphic for every $t\in K$ such that $t(t^2+1)(2t^2+1)(t^2+2)(t^2-1)(t^4+t^2+1)\neq 0$.\\
\textit{Case 2:} If $\text{char}(K)=5$, the polynomials
\begin{align*}
R_2(t)&:=\dfrac{J_4(t)J_2(-t)^2-J_4(-t)J_2(t)^2}{t(t^2+1)^3(2t^2+1)^2(t^2+2)^2(t^4+t^2+1)^2},\\
R_3(t)&:=\dfrac{J_6(t)J_2(-t)^3-J_6(-t)J_2(t)^3}{t^5(t^2+1)^3(2t^2+1)^3(t^2+2)^3(t^4+t^2+1)^2},\\
R_5(t)&:=\dfrac{J_{10}(t)J_2(-t)^5-J_{10}(-t)J_2(t)^5}{t^5(t^2+1)^7(2t^2+1)^5(t^2+2)^5},
\end{align*}
are defined over $\mathbb{F}_5[t]$. Moreover, $\text{gcd}(R_2,R_3)=\text{gcd}(R_2,R_5)=1$. Thus, the curves $C_t$ and $C_{-t}$ are not isomorphic for every $t\in K$ such that $t(t^2+1)(2t^2+1)(t^2+2)(t^2-1)(t^4+t^2+1)\neq 0$.\\
\textit{Case 3:} If $\text{char}(K)=7$, the polynomials
\begin{align*}
R_2(t)&:=\dfrac{J_4(t)J_2(-t)^2-J_4(-t)J_2(t)^2}{t(t^2+1)(2t^2+1)^2(t^2+2)^2},\\
R_3(t)&:=\dfrac{J_6(t)J_2(-t)^3-J_6(-t)J_2(t)^3}{t^5(t^2+1)^3(2t^2+1)^2(t^2+2)^2},\\
R_5(t)&:=\dfrac{J_{10}(t)J_2(-t)^5-J_{10}(-t)J_2(t)^5}{t^5(t^2+1)^3(2t^2+1)^2(t^2+2)^2},
\end{align*}
are defined over $\mathbb{F}_7[t]$. Moreover, $\text{gcd}(R_2,R_3)=\text{gcd}(R_2,R_5)=1$. Thus, the curves $C_t$ and $C_{-t}$ are not isomorphic for every $t\in K$ such that $t(t^2+1)(2t^2+1)(t^2+2)(t^2-1)(t^4+t^2+1) \neq 0$.\\
\textit{Case 4:} If $\text{char}(K)=11$, the polynomials
\begin{align*}
R_2(t)&:=\dfrac{J_4(t)J_2(-t)^2-J_4(-t)J_2(t)^2}{t(t^2+1)(2t^2+1)(t^2+2)(t^4 + 5t^2 + 1)^2},\\
R_3(t)&:=\dfrac{J_6(t)J_2(-t)^3-J_6(-t)J_2(t)^3}{t^5(t^2+1)(2t^2+1)(t^2+2)(t^4 + 5t^2 + 1)^3},\\
R_5(t)&:=\dfrac{J_{10}(t)J_2(-t)^5-J_{10}(-t)J_2(t)^5}{t^5(t^2+1)^3(2t^2+1)(t^2+2)(t^4 + 5t^2 + 1)^5},
\end{align*}
are defined over $\mathbb{F}_{11}[t]$. It turns out that $\text{gcd}(J_2(t),J_2(-t))=t^4 + 5t^2 + 1$. To circumvent this problem, other weighted polynomials can be defined as follows.
\begin{align*}
R_{23}(t)&:=\dfrac{J_4(t)^3J_6(-t)^2-J_4(-t)^3J_6(t)^2}{t^{11}(t^2+1)(2t^2+1)(t^2+2)},\\
R_{35}(t)&:=\dfrac{J_6(t)^5J_{10}(-t)^3-J_6(-t)^5J_{10}^3(t)}{t^{41}(t^2+1)^7(2t^2+1)(t^2+2)},\\
R_{25}(t)&:=\dfrac{J_{10}^2(t)J_4^5(-t)-J_{10}^2(-t)J_4^5(t)}{t^{11}(t^2+1)^5(2t^2+1)(t^2+2)},
\end{align*}
which are all still defined over $\mathbb{F}_{11}[t]$. Moreover, $\text{gcd}(R_{23},R_{35})=\text{gcd}(R_{25},R_{35})=\text{gcd}(R_{23},R_{25})=1$. Thus, the curves $C_t$ and $C_{-t}$ are indeed non-isomorphic for every $t\in K$ such that $t(t^2+1)(2t^2+1)(t^2+2)(t^2-1)(t^4+t^2+1)$.\\
\textit{Case 5:} If $\text{char}(K)=23$, the polynomials
\begin{align*}
R_2(t)&:=\dfrac{J_4(t)J_2(-t)^2-J_4(-t)J_2(t)^2}{t(t^2+1)(2t^2+1)(t^2+2)(t^2 + 13)(t^2+16)},\\
R_3(t)&:=\dfrac{J_6(t)J_2(-t)^3-J_6(-t)J_2(t)^3}{t^5(t^2+1)(2t^2+1)(t^2+2)(t^2 + 13)(t^2+16)},\\
R_5(t)&:=\dfrac{J_{10}(t)J_2(-t)^5-J_{10}(-t)J_2(t)^5}{t^5(t^2+1)^3(2t^2+1)(t^2+2)(t^2 + 13)(t^2+16)},
\end{align*}
are defined over $\mathbb{F}_{23}[t]$. On the other hand, if $t_0^2 \in \{10,7\}$, the  curves $C_t$ and $C_{-t}$ are both isomorphic to the curve
\begin{equation*}
C:y^2=x^6 + x^3 + 2.
\end{equation*}
Thus, the curves $C_t$ and $C_{-t}$ are indeed non-isomorphic for every $t\in K$ such that $t(t^2+1)(2t^2+1)(t^2+2)(t^2-1)(t^4+t^2+1)(t^2+13)(t^2+16)$.\\
\textit{Case 6:} If $\text{char}(K)=37$, the polynomials
\begin{align*}
R_2(t)&:=\dfrac{J_4(t)J_2(-t)^2-J_4(-t)J_2(t)^2}{t(t^2+1)(2t^2+1)(t^2+2)(t^2+5)^2(t^2+15)^2},\\
R_3(t)&:=\dfrac{J_6(t)J_2(-t)^3-J_6(-t)J_2(t)^3}{t^5(t^2+1)(2t^2+1)(t^2+2)(t^2+5)^3(t^2+15)^3},\\
R_5(t)&:=\dfrac{J_{10}(t)J_2(-t)^5-J_{10}(-t)J_2(t)^5}{t^5(t^2+1)^3(2t^2+1)(t^2+2)(t^2+5)^5(t^2+15)^5},
\end{align*}
are defined over $\mathbb{F}_{37}[t]$. Yet again $\text{gcd}(J_2(t),J_2(-t))=(t^2+5)(t^2+15)$. Nonetheless, the polynomials
\begin{align*}
R_{23}(t)&:=\dfrac{J_4(t)^3J_6(-t)^2-J_4(-t)^3J_6(t)^2}{t^{11}(t^2+1)(2t^2+1)(t^2+2)},\\
R_{35}(t)&:=\dfrac{J_6(t)^5J_{10}(-t)^3-J_6(-t)^5J_{10}^3(t)}{t^{41}(t^2+1)^7(2t^2+1)(t^2+2)},\\
R_{25}(t)&:=\dfrac{J_{10}^2(t)J_4^5(-t)-J_{10}^2(-t)J_4^5(t)}{t^{11}(t^2+1)^5(2t^2+1)(t^2+2)},
\end{align*}
are defined over $\mathbb{F}_{37}[t]$. Moreover, $\text{gcd}(R_{23},R_{35})=\text{gcd}(R_{25},R_{35})=\text{gcd}(R_{23},R_{25})=1$. Thus, the curves $C_t$ and $C_{-t}$ are not isomorphic for every $t\in K$ such that $t(t^2+1)(2t^2+1)(t^2+2)(t^2-1)(t^4+t^2+1)\neq 0$.\\
\textit{Case 7:} If $\text{char}(K)=47$, the polynomials
\begin{align*}
R_2(t)&:=\dfrac{J_4(t)J_2(-t)^2-J_4(-t)J_2(t)^2}{t(t^2+1)(2t^2+1)(t^2+2)(t^2+26)(t^2+38)},\\
R_3(t)&:=\dfrac{J_6(t)J_2(-t)^3-J_6(-t)J_2(t)^3}{t^5(t^2+1)(2t^2+1)(t^2+2)(t^2+26)(t^2+38)(t^2+3)(t^2+16)},\\
R_5(t)&:=\dfrac{J_{10}(t)J_2(-t)^5-J_{10}(-t)J_2(t)^5}{t^5(t^2+1)^3(2t^2+1)(t^2+2)(t^2+26)(t^2+38)(t^2+3)(t^2+16)},
\end{align*}
are defined over $\mathbb{F}_{47}[t]$. On the other hand if $t^2 \in \{26,38\}$, the  curves $C_t$ and $C_{-t}$ are both isomorphic to the curve
\begin{equation*}
C:y^2=x^6 + 16x^5 + 41x^4 + 4x^3 + 41x^2 + 16x + 1.
\end{equation*}
Thus, the curves $C_t$ and $C_{-t}$ are not isomorphic for every $t\in K$ such that $t(t^2+1)(2t^2+1)(t^2+2)(t^2-1)(t^4+t^2+1)(t^2+26)(t^2+38)\neq 0$.

\begin{remark}
Roots of the polynomial $t(t^2+1)(t^2-t+1)$ coincide with zeros of $J_{10}(t)$. Once again the first two factors correspond to a cusp on the modular curve $X_0(4)$. The geometric isomorphism classes of $E_t$ and $E'_t$ coincide, if $t$ is a  root of $t^2 - t + 1$.  On the other hand, roots of $(2t^2+1)(t^2+2)$ correspond to isomorphic pairs of Curves.
\end{remark}

\section{Proof of Theorem 3}\label{example3}

Similar to the proof of Theorem 1, we use the parameterization above of the rational curve $X_0(7)$ to construct a one-parameter family of pairs of $3$-isogenous elliptic curves: We denote by
\begin{align*}
&(E_s, E'_s) \text{, a pair of $7$-isogeneous elliptic curves,}\\
&J_7(s), \text{ resp. } j'_7(s) \text{, the j-invariant of $E_s$, resp. $E'_s$,}\\
&\Delta_7(s), \text{ resp. } \Delta'_7(s)   \text{, the discriminant of $E_s$, resp. $E'_s$.}\\
\end{align*}
In the case at hand, \footnote{We used a similar Magma routine as in the proof of theorem $1$.}
\begin{align*}
j_7(s)=&\dfrac{(s^2+13s+49)(s^2+5s+1)^3}{s},\\
E_s:y^2 &= x^3 + \frac{1}{4}(s^4 + 14s^3 + 63s^2 + 70s - 7)x^2 - 36sx -s (s^4 + 14s^3 + 63s^2 + 70s - 7), \\
\Delta_7(s)=&s(s^2 + 5s + 1)^6(s^2 + 13s + 49)^2,\\
j'_7(s)=&\dfrac{(s^2 + 13s + 49)(s^2 + 245s + 2401)^3}{s^7},\\
E'_s:y^2&= x^3+ax^2+bx+c;\\
a:&=\frac{1}{4}(s^4 + 14s^3 + 63s^2 + 70s - 7),\\
b:&=-( 5s^7 + 165s^6 +
    2180s^5 + 14555s^4 + 49820s^3 + 75215s^2 + 25431s + 2450),\\
c:&= -s^{11}
    - 61s^{10} - 1563s^9 - 22420s^8 - 199153s^7 - 1132425s^6 -
    4079892s^5 - 8795374s^4\\
    & \ \ \ \   - 9879408s^3 - 4152015s^2 - 725788s - 45276,\\
\Delta'_7(s)=&s^7(s^2 + 5s + 1)^6(s^2 + 13s + 49)^2.\\
\end{align*}

The Galois restrictions in the odd degree imposes the condition $s=t^2$, for some $t \in K$. Let $r_i$, resp. $s_i$ be the $x$-coordinate of the Weierstrass points of $E_t$, resp. $E'_t$, enumerated such that $\psi((r_i,0))=(s_i,0)$. Using the same notation and method as in the proof of theorem 1, we used Magma \cite{Mag} to compute the following:
\begin{align*}
\gamma_{21}\gamma_{13}\gamma_{32}&=\tilde{\gamma}_{21}\tilde{\gamma}_{13}\tilde{\gamma}_{32}=t^6\\
\gamma_{21}+\gamma_{13}+\gamma_{32}&=\dfrac{1}{16t} (t^8 - 8t^7 + 38t^6 - 128t^5 + 327t^4 - 640t^3 + 910t^2 - 784t - 343), \\
\tilde{\gamma}_{21}+\tilde{\gamma}_{13}+\tilde{\gamma}_{32}&=-\dfrac{1}{16t}(t^8 + 8t^7 + 38t^6 + 128t^5 + 327t^4 + 640t^3 + 910t^2 + 784t - 343) ,\\
\gamma_{21}\gamma_{13}+\gamma_{32}\gamma_{21}+\gamma_{21}\gamma_{13}&=-\dfrac{1}{16t}(t^8 + 16t^7 - 130t^6 + 640t^5 - 2289t^4 + 6272t^3 - 13034t^2 + 19208t -
    16807) ,\\
\tilde{\gamma}_{21}\tilde{\gamma}_{13}+\tilde{\gamma}_{32}\tilde{\gamma}_{21}+\tilde{\gamma}_{21}\tilde{\gamma}_{13}&=\dfrac{1}{16t}(t^8 - 16t^7 - 130t^6 - 640t^5 - 2289t^4 - 6272t^3 - 13034t^2 - 19208t -
    16807) ,\\
    \end{align*}
    \begin{align*}
    \kappa=&2^{-10}t^{17}(t^2 - 5t + 7)^3(t^2 - 3t + 7)^3(t^2 - t + 7)^8(t^2 + t + 7)^8(t^4 + 5t^2 + 1)^{21} ,\\
\tilde{\kappa}=&- 2^{-10}t^{17}(t^2 + 5t + 7)^3(t^2 + 3t + 7)^3(t^2 - t + 7)^8(t^2 + t + 7)^8(t^4 + 5t^2 + 1)^{21},\\
h_t=&2^{-10}t^{16}(t^2 - 5t + 7)^3(t^2 - 3t + 7)^3(t^2 - t + 7)^8(t^2 + t + 7)^8(t^4 + 5t^2 + 1)^{21} \\
&\cdot(16 t^7 x^6 + (t^8 + 16 t^7 - 130 t^6 + 640 t^5 - 2289 t^4 + 6272 t^3 - 13034 t^2
    + 19208 t - 16807) x^4 \\
    &+ (t^8 - 8 t^7 + 38 t^6 - 128 t^5 + 327 t^4 - 640 t^3
    + 910 t^2 - 784 t - 343) x^2 - 16 t),\\
\tilde{h}_t=&2^{-10}t^{16}(t^2 + 5t + 7)^3(t^2 + 3t + 7)^3(t^2 - t + 7)^8(t^2 + t + 7)^8(t^4 + 5t^2 + 1)^{21}  \\
&\cdot(-16 t^7 x^6 + (t^8 - 16 t^7 - 130 t^6 - 640 t^5 - 2289 t^4 - 6272 t^3 - 13034 t^2
    - 19208 t - 16807) x^4 \\
    &+ (t^8 + 8 t^7 + 38 t^6 + 128 t^5 + 327 t^4 + 640 t^3
    + 910 t^2+ 784 t - 343) x^2 + 16 t)
\end{align*}
%\begin{align*}
%a_2&=81(t^2 - 5t + 7)(t^2 - 3t + 7)(t^2 - t + 7)(t^2 + t + 7)(t^4 + 5t^2 + 1)^2 ,\\
%\tilde{a_2}&=-81(t^2 + 5t + 7)(t^2 + 3t + 7)(t^2 + t + 7)(t^2 - t + 7)(t^4 + 5t^2 + 1)^2   %,
%\end{align*}
Let $C_t$ be the genus $2$ hyperelliptic curve defined by the equation
\begin{align*}
C_t:&(t^2-5t+7)(t^2-3t+7)(t^4 + 5t^2 + 1)y^2=\\
&(16 t^7 x^6 + (t^8 + 16 t^7 - 130 t^6 + 640 t^5 - 2289 t^4 + 6272 t^3 - 13034 t^2
    + 19208 t - 16807) x^4 \\
    &+ (t^8 - 8 t^7 + 38 t^6 - 128 t^5 + 327 t^4 - 640 t^3
    + 910 t^2 - 784 t - 343) x^2 - 16 t).\\
\end{align*}
The curve $y^2=h_t(x)$, resp. $y^2=\tilde{h}_t(x)$ is isomorphic to $C_t$, resp. $C_{-t}$, over $K$. As before, the pair $C_t$ and $C_{-t}$ have isomorphic unpolarized Jacobian for any $t \in K$ as long as $ t(t^4+13t^2+49)(t^8 - 6t^6 + 43t^4 - 294t^2 + 2401)(t^4 + 5t^2 + 1)(t^2+7)(t^4 + 245t^2 + 2401)\neq 0.$\\

Finally, let's investigate geometric isomorphism classes of $C_t$ and $C_{-t}$. Let $J_2(t)$, $J_4(t)$, $J_6(t)$, $J_8(t)$,  $J_{10}(t) \in  \mathbb{Z}(t) $ be  Igusa invariants of the twist $y^2=(16 t^7 x^6 + (t^8 + 16 t^7 - 130 t^6 + 640 t^5 - 2289 t^4 + 6272 t^3 - 13034 t^2
    + 19208 t - 16807) x^4 + (t^8 - 8 t^7 + 38 t^6 - 128 t^5 + 327 t^4 - 640 t^3
    + 910 t^2 - 784 t - 343) x^2 - 16 t)$. The polynomials
\begin{align*}
R_2(t)&:=\frac{J_4(t)J_2(-t)^2-J_4(-t)J_2(t)^2}{t (t^4+13 t^2+49)^2 (t^2+7) (t^4 + 5 t^2 + 1) (t^4 + 245 t^2 + 2401)},\\
R_3(t)&:=\frac{J_6(t)J_2(-t)^3-J_6(-t)J_2(t)^3}{t^9 (t^4+13 t^2+49)^2 (t^2+7) (t^4 + 5 t^2 + 1) (t^4 + 245 t^2 + 2401)},\\
R_5(t)&:=\frac{J_{10}(t)J_2(-t)^5-J_{10}(-t)J_2(t)^5}{t^9 (t^4+13 t^2+49)^4 (t^2+7) (t^4 + 5 t^2 + 1) (t^4 + 245 t^2 + 2401)}.
\end{align*}
are defined over $\mathbb{Z}[t]$. Using Magma \cite{Mag} one finds that, $\text{gcd}(\text{resultant}(R_2,R_3),\text{resultant}(R_2,R_5))$ is divisible exactly by $2,3,5,7,13,17,19,41,167$ and $571603$. So unless $\text{char}(K)$ is one of those, the curves $C_t$ and $C_{-t}$ are not isomorphic for every $t\in K$ such that $ t(t^4+13t^2+49)(t^8 - 6t^6 + 43t^4 - 294t^2 + 2401)(t^8+14t^6+63t^4+70t^2-7)(t^4 + 5t^2 + 1)(t^2+7)(t^4 + 245t^2 + 2401)\neq 0$.\\
\textit{Case 1:} If $\text{char}(K)=3$, the polynomials
\begin{align*}
R_2(t)&:=\dfrac{J_4(t)J_2(-t)^2-J_4(-t)J_2(t)^2}{t (t^4+13 t^2+49)^5 (t^2+7)^9 (t^4 + 5 t^2 + 1) (t^4 + 245 t^2 + 2401)},\\
R_3(t)&:=\dfrac{J_6(t)J_2(-t)^3-J_6(-t)J_2(t)^3}{t^9 (t^4+13 t^2+49)^3 (t^2+7)^11 (t^4 + 5 t^2 + 1) (t^4 + 245 t^2 + 2401)},\\
R_5(t)&:=\dfrac{J_{10}(t)J_2(-t)^5-J_{10}(-t)J_2(t)^5}{t^9 (t^4+13 t^2+49)^19 (t^2+7)^29 (t^4 + 5 t^2 + 1) (t^4 + 245 t^2 + 2401)},
\end{align*}
are defined over $\mathbb{F}_3[t]$. Moreover, $\text{gcd}(R_2,R_3)=\text{gcd}(R_2,R_5)=1$. Thus, the curves $C_t$ and $C_{-t}$ are not isomorphic for every $t\in K$ such that $ t(t^4+13t^2+49)(t^8 - 6t^6 + 43t^4 - 294t^2 + 2401)(t^8+14t^6+63t^4+70t^2-7)(t^4 + 5t^2 + 1)(t^2+7)(t^4 + 245t^2 + 2401)\neq 0$.\\
\textit{Case 2:}  If $\text{char}(K)=5$, the polynomials
\begin{align*}
R_2(t)&:=\dfrac{J_4(t)J_2(-t)^2-J_4(-t)J_2(t)^2}{t (t^4+13 t^2+49)^4 (t^2+7) (t^4 + 5 t^2 + 1)^5 (t^4 + 245 t^2 + 2401)},\\
R_3(t)&:=\dfrac{J_6(t)J_2(-t)^3-J_6(-t)J_2(t)^3}{t^9 (t^4+13 t^2+49)^2 (t^2+7) (t^4 + 5 t^2 + 1)^5 (t^4 + 245 t^2 + 2401)},\\
R_5(t)&:=\dfrac{J_{10}(t)J_2(-t)^5-J_{10}(-t)J_2(t)^5}{t^9 (t^4+13 t^2+49)^5 (t^2+7)^3 (t^4 + 5 t^2 + 1)^9 (t^4 + 245 t^2 + 2401)},
\end{align*}
are defined over $\mathbb{F}_5[t]$. Moreover, $\text{gcd}(R_2,R_3)=\text{gcd}(R_2,R_5)=1$. Thus, the curves $C_t$ and $C_{-t}$ are not isomorphic for every $t\in K$ such that $ t(t^4+13t^2+49)(t^8 - 6t^6 + 43t^4 - 294t^2 + 2401)(t^8+14t^6+63t^4+70t^2-7)(t^4 + 5t^2 + 1)(t^2+7)(t^4 + 245t^2 + 2401)\neq 0$.\\
\textit{Case 3:} If $\text{char}(K)=7$, the polynomials
\begin{align*}
R_2(t)&:=\dfrac{J_4(t)J_2(-t)^2-J_4(-t)J_2(t)^2}{t^{21} (t^4+13 t^2+49)^5 (t^4 + 245 t^2 + 2401)},\\
R_3(t)&:=\dfrac{J_6(t)J_2(-t)^3-J_6(-t)J_2(t)^3}{t^{39} (t^4+13 t^2+49)^4 (t^4 + 245 t^2 + 2401)},\\
R_5(t)&:=\dfrac{J_{10}(t)J_2(-t)^5-J_{10}(-t)J_2(t)^5}{t^{67} (t^4+13 t^2+49)^6 (t^4 + 245 t^2 + 2401)},
\end{align*}
are defined over $\mathbb{F}_7[t]$. Moreover, $\text{gcd}(R_2,R_3)=\text{gcd}(R_2,R_5)=1$. Thus, the curves $C_t$ and $C_{-t}$ are not isomorphic for every $t\in K$ such that $ t(t^4+13t^2+49)(t^8 - 6t^6 + 43t^4 - 294t^2 + 2401)(t^8+14t^6+63t^4+70t^2-7)(t^4 + 5t^2 + 1)(t^2+7)(t^4 + 245t^2 + 2401)\neq 0$.\\
\textit{Case 4:} If $\text{char}(K)=13$, the polynomials
\begin{align*}
R_2(t)&:=\dfrac{J_4(t)J_2(-t)^2-J_4(-t)J_2(t)^2}{t (t^4+13 t^2+49)^2 (t^2+7) (t^4 + 5 t^2 + 1) (t^4 + 245 t^2 + 2401) (t^2+6)^2},\\
R_3(t)&:=\dfrac{J_6(t)J_2(-t)^3-J_6(-t)J_2(t)^3}{t^9 (t^4+13 t^2+49)^2 (t^2+7) (t^4 + 5 t^2 + 1) (t^4 + 245 t^2 + 2401) (t^2+6)^2},\\
R_5(t)&:=\dfrac{J_{10}(t)J_2(-t)^5-J_{10}(-t)J_2(t)^5}{t^9 (t^4+13 t^2+49)^5 (t^2+7) (t^4 + 5 t^2 + 1) (t^4 + 245 t^2 + 2401) (t^2+6)^2},
\end{align*}
are defined over $\mathbb{F}_{13}[t]$. On the other hand, if $t^2 =-6$, the  curves $C_t$ and $C_{-t}$ are both isomorphic to the curve
\begin{equation*}
C:y^2 = x^5 + x^3 + 8 x. 
\end{equation*}
Thus, the curves $C_t$ and $C_{-t}$ are not isomorphic for every $t\in K$ such that $ t(t^4+13t^2+49)(t^8 - 6t^6 + 43t^4 - 294t^2 + 2401)(t^8+14t^6+63t^4+70t^2-7)(t^4 + 5t^2 + 1)(t^2+7)(t^4 + 245t^2 + 2401)(t^2+6)\neq 0$.\\
\textit{Case 5:} If $\text{char}(K)=17$, the polynomials
\begin{align*}
R_2(t)&:=\dfrac{J_4(t)J_2(-t)^2-J_4(-t)J_2(t)^2}{t (t^4+13 t^2+49)^2 (t^2+7) (t^4 + 5 t^2 + 1) (t^4 + 245 t^2 + 2401) (t^4 + 11 t^2 + 15) (t^4 + 7 t^2 + 15)},\\
R_3(t)&:=\dfrac{J_6(t)J_2(-t)^3-J_6(-t)J_2(t)^3}{t^9 (t^4+13 t^2+49)^2 (t^2+7) (t^4 + 5 t^2 + 1) (t^4 + 245 t^2 + 2401) (t^4 + 11 t^2 + 15) (t^4 + 7 t^2 + 15)},\\
R_5(t)&:=\dfrac{J_{10}(t)J_2(-t)^5-J_{10}(-t)J_2(t)^5}{t^{11} (t^4+13 t^2+49)^4 (t^2+7) (t^4 + 5 t^2 + 1) (t^4 + 245 t^2 + 2401) (t^4 + 11 t^2 + 15) (t^4 + 7 t^2 + 15)},
\end{align*}
are defined over $\mathbb{F}_{17}[t]$. On the other hand if $(t^4 + 11 t^2 + 15) (t^4 + 7 t^2 + 15)=0$, the  curves $C_t$ and $C_{-t}$ are both isomorphic to the curve
\begin{equation*}
C:y^2 = x^5 + x^3 + 7 x .
\end{equation*}
Thus, the curves $C_t$ and $C_{-t}$ are not isomorphic for every $t\in K$, such that $ t(t^4+13t^2+49)(t^8 - 6t^6 + 43t^4 - 294t^2 + 2401)(t^8+14t^6+63t^4+70t^2-7)(t^4 + 5t^2 + 1)(t^2+7)(t^4 + 245t^2 + 2401)(t^4 + 11 t^2 + 15) (t^4 + 7 t^2 + 15)\neq 0$.\\
\textit{Case 6:} If $\text{char}(K)=19$, the polynomials
\begin{align*}
R_2(t)&:=\dfrac{J_4(t)J_2(-t)^2-J_4(-t)J_2(t)^2}{t (t^4+13 t^2+49)^2 (t^2+7)^3 (t^4 + 5 t^2 + 1) (t^4 + 245 t^2 + 2401)},\\
R_3(t)&:=\dfrac{J_6(t)J_2(-t)^3-J_6(-t)J_2(t)^3}{t^9 (t^4+13 t^2+49)^2 (t^2+7)^3 (t^4 + 5 t^2 + 1) (t^4 + 245 t^2 + 2401)},\\
R_5(t)&:=\dfrac{J_{10}(t)J_2(-t)^5-J_{10}(-t)J_2(t)^5}{t^9 (t^4+13 t^2+49)^4 (t^2+7)^5 (t^4 + 5 t^2 + 1) (t^4 + 245 t^2 + 2401)},
\end{align*}
are defined over $\mathbb{F}_{19}[t]$. Moreover, $\text{gcd}(R_2,R_3)=\text{gcd}(R_2,R_5)=1$. Thus, the curves $C_t$ and $C_{-t}$ are not isomorphic for every $t\in K$ such that $ t(t^4+13t^2+49)(t^8 - 6t^6 + 43t^4 - 294t^2 + 2401)(t^8+14t^6+63t^4+70t^2-7)(t^4 + 5t^2 + 1)(t^2+7)(t^4 + 245t^2 + 2401)\neq 0$.\\
\textit{Case 7:}  If $\text{char}(K)=41$, the polynomials
\begin{align*}
R_2(t)&:=\dfrac{J_4(t)J_2(-t)^2-J_4(-t)J_2(t)^2}{t (t^4+13 t^2+49)^2 (t^2+7) (t^4 + 5 t^2 + 1) (t^4 + 245 t^2 + 2401) (t^4 + 26 t^2 + 8)},\\
R_3(t)&:=\dfrac{J_6(t)J_2(-t)^3-J_6(-t)J_2(t)^3}{t^9 (t^4+13 t^2+49)^2 (t^2+7) (t^4 + 5 t^2 + 1) (t^4 + 245 t^2 + 2401) (t^4 + 26 t^2 + 8)},\\
R_5(t)&:=\dfrac{J_{10}(t)J_2(-t)^5-J_{10}(-t)J_2(t)^5}{t^9 (t^4+13 t^2+49)^4 (t^2+7) (t^4 + 5 t^2 + 1) (t^4 + 245 t^2 + 2401) (t^4 + 26 t^2 + 8)},
\end{align*}
are defined over $\mathbb{F}_{47}[t]$. On the other hand, if $(t^4 + 26 t^2 + 8)=0$, the  curves $C_t$ and $C_{-t}$ are both isomorphic to the curve
\begin{equation*}
C:y^2 = x^5 + x^3 + 14 x.
\end{equation*}
Thus, the curves $C_t$ and $C_{-t}$ are not isomorphic for every $t\in K$ such that $ t(t^4+13t^2+49)(t^8 - 6t^6 + 43t^4 - 294t^2 + 2401)(t^8+14t^6+63t^4+70t^2-7)(t^4 + 5t^2 + 1)(t^2+7)(t^4 + 245t^2 + 2401)(t^4 + 26 t^2 + 8)\neq 0$.\\
\textit{Case 8:} If $\text{char}(K)=167$, the polynomials
\begin{align*}
R_2(t)&:=\dfrac{J_4(t)J_2(-t)^2-J_4(-t)J_2(t)^2}{t (t^4+13 t^2+49)^2 (t^2+7)^3 (t^4 + 5 t^2 + 1) (t^4 + 245 t^2 + 2401)},\\
R_3(t)&:=\dfrac{J_6(t)J_2(-t)^3-J_6(-t)J_2(t)^3}{t^9 (t^4+13 t^2+49)^2 (t^2+7)^3 (t^4 + 5 t^2 + 1) (t^4 + 245 t^2 + 2401)},\\
R_5(t)&:=\dfrac{J_{10}(t)J_2(-t)^5-J_{10}(-t)J_2(t)^5}{t^9 (t^4+13 t^2+49)^4 (t^2+7)^5 (t^4 + 5 t^2 + 1) (t^4 + 245 t^2 + 2401)},
\end{align*}
are defined over $\mathbb{F}_{167}[t]$. Moreover, $\text{gcd}(R_2,R_3)=\text{gcd}(R_2,R_5)=1$. Thus, the curves $C_t$ and $C_{-t}$ are not isomorphic for every $t\in K$ such that $ t(t^4+13t^2+49)(t^8 - 6t^6 + 43t^4 - 294t^2 + 2401)(t^8+14t^6+63t^4+70t^2-7)(t^4 + 5t^2 + 1)(t^2+7)(t^4 + 245t^2 + 2401)\neq 0$.\\
\textit{Case 9:}  If $\text{char}(K)=571603$, the polynomials
\begin{align*}
R_2(t)&:=\dfrac{J_4(t)J_2(-t)^2-J_4(-t)J_2(t)^2}{t (t^4+13 t^2+49)^2 (t^2+7) (t^4 + 5 t^2 + 1) (t^4 + 245 t^2 + 2401) (t^4 + 516817 t^2 + 49)},\\
R_3(t)&:=\dfrac{J_6(t)J_2(-t)^3-J_6(-t)J_2(t)^3}{t^9 (t^4+13 t^2+49)^2 (t^2+7) (t^4 + 5 t^2 + 1) (t^4 + 245 t^2 + 2401) (t^4 + 516817 t^2 + 49)^2},\\
R_5(t)&:=\dfrac{J_{10}(t)J_2(-t)^5-J_{10}(-t)J_2(t)^5}{t^9 (t^4+13 t^2+49)^4 (t^2+7) (t^4 + 5 t^2 + 1) (t^4 + 245 t^2 + 2401) (t^4 + 516817 t^2 + 49)^4},
\end{align*}
are defined over $\mathbb{F}_{571603}[t]$. Moreover, $\text{gcd}(R_2,R_3)=\text{gcd}(R_2,R_5)=1$. Thus, the curves $C_t$ and $C_{-t}$ are not isomorphic for every $t\in K$ such that $ t(t^4+13t^2+49)(t^8 - 6t^6 + 43t^4 - 294t^2 + 2401)(t^8+14t^6+63t^4+70t^2-7)(t^4 + 5t^2 + 1)(t^2+7)(t^4 + 245t^2 + 2401)\neq 0$.\\

\begin{remark}
Roots of the polynomial $t(t^2-3t+7)(t^2-5t+7)(t^4+13t^2+49)$ coincide with zeros of $J_{10}(t)$. In this case the first, resp. second, factor correspond to a cusp, resp. elliptic point, on the modular curve $X_0(7)$. The geometric isomorphism classes of $E_t$ and $E'_t$ coincide, if $t$ is a  root of $(t^2-3t+7)(t^2-5t+7)(t^2+7)$. Nonetheless, roots of $(t^2+7)(t^4+5t^2+1)(t^4+245t^2+2401)$ correspond to well defined, yet isomorphic pairs of Curves.
\end{remark}

\section{Obstruction to constructing further examples over $\mathbb{Q}$}\label{s2o}

One might be tempted to replicate the above construction to isogenies of higher degrees. Unfortunately, an obstruction emerges from the imposed Galois restrictions, when considering a pair of $n$-cyclically-isogenous elliptic curves for $n \geq 5$, except for $n=7$. Recall that for $n$ odd, the Galois restriction took the form of a requirement on one- hence on both- elliptic curve to have a square discriminant. For even degrees, the isogenous elliptic curves were required to have discriminants varying by a square. It turns out that these restrictions are far too strong over $\mathbb{Q}$, when $n \notin \{ 2, 3, 4, 7\}$. Let's first investigate such restrictions for odd degrees.

\textit{Degree 5:}
For a number field $K$, and $s \in K$, the isomorphism class corresponding to the $j$-invariant
 \begin{equation*}
j_5(s)=\dfrac{(s^2+10s+5)^3}{s}
\end{equation*}
contains an elliptic curve, which can be defined over $K$, and admits a $K$-rational cyclic isogeny of degree 5 that can be also defined over $K$. Bear in mind that the discriminants differ by a square for isomorphic elliptic curves. We used Magma's function \textsf{ EllipticCurveFromjInvariant } to compute an elliptic curve in that isomorphism class, and the magma function  \textsf{ Discriminant } to compute its discriminant, $\Delta_5(s)$. Up to a square,
\begin{equation*}
\Delta_5(s)=s^2 + 22s + 125.
\end{equation*}
Values for $s$ for which $\Delta_5(s)$ is a square are in one to one correspondence with the rational points on the elliptic curve
\begin{equation*}
O_5:y^2=x^3 + 22x^2 + 125x.
\end{equation*}
$O_5(\mathbb{Q})$ contains only one rational point, namely $(0,0)$.\\

In what follows, the same procedure and similar notation will be used, while varying the degree from $5$ to $n \in \{9,13,25 \}$.\\

\textit{Degree 9:}
In the isomorphism class
\begin{equation*}
j_9(s)=\dfrac{(s+3)^3(s^3+9s^2+27s+3)^3}{s(s^2+9s+27)},
\end{equation*}
an elliptic curve, up to a square, has discriminant:
\begin{equation*}
\Delta_9(s)=s(s^2 + 9s + 27).
\end{equation*}
Values for $s$ for which $\Delta_9(s)$ is a square are in one to one correspondence with the rational points on the elliptic curve
\begin{equation*}
O_9:y^2=x^3 + 9x^2 + 27x.
\end{equation*}
$O_9(\mathbb{Q})$ contains only one rational point, namely $(0,0)$.\\

\textit{Degree 13:}
In the isomorphism class
\begin{equation*}
j_{13}(s)=\dfrac{(s^2+5s+13)(s^4+7s^3+20s^2+19s+1)^3}{s},
\end{equation*}
an elliptic curve, up to a square, has discriminant:
\begin{equation*}
\Delta_{13}(s)=s(s^2 + 6s + 13).
\end{equation*}
Values for $s$ for which $\Delta_{13}(s)$ is a square are in one-to-one correspondence with the rational points on the elliptic curve
\begin{equation*}
O_{13}:y^2=x^3 + 6x^2 + 13x.
\end{equation*}
$O_{13}(\mathbb{Q})$ contains only one rational point, namely $(0,0)$.\\

\textit{Degree 25:}
In the isomorphism class
\begin{equation*}
j_{25}(s)=\dfrac{(s^{10}+10s^9+55s^8+200s^7+525s^6+1010s^5+1425s^4+1400s^3+875s^2+250s+5)^3}{s(s^4+5s^3+15s^2+25s+25)},
\end{equation*}
an elliptic curve, up to a square, has discriminant:
\begin{align*}
\Delta_{25}(s)=&s(s^4 + 5s^3 + 15s^2 + 25s + 25)(s^2 + 2s + 5).
\end{align*}
Values for $s$ for which $\Delta_{25}(s)$ is a square are in one-to-one correspondence with the rational points on the genus $3$ hyperelliptic curve
\begin{equation*}
O_{25}:y^2=x(x^4 + 5x^3 + 15x^2 + 25x + 25)(x^2 + 2x + 5).
\end{equation*}
By Faltings theorem, see \cite{Falt}, the hyperelliptic curve $O_{25}$ has only finitely many rational points over $\mathbb{Q}$. \\

In considering Galois restrictions for isogenies of even degree, it is necessary to compute both isogenous curves in order to compare their discriminants. Since the discriminant vary by a square within the same geometric isomorphism class, it suffices to compute any pair of elliptic curves in the isomorphism classes $j_n$ and $j'_n$, for $n \in \{6,8,10,12,16,18\}$.\\

\textit{Degree 6:}
As before, we used Magma, in order to compute a pair of elliptic curves in the geometric isomorphism classes
\begin{equation*}
j_{6}(s)=\dfrac{(s+6)^3(s^3+18s^2+84s+24)^3}{s(s+8)^3(s+9)^2}, \  j'_{6}(s)=\dfrac{(s+12)^3(s^3+252s^2+3888s+15552)^3}{s^6(s+8)^2(s+9)^3}.
\end{equation*}
Up to a square, the respective discriminants of which turned out to be 
\begin{equation*}
\Delta_{6}(s)=s(s+8), \ \Delta'_{6}(s)=(s+9).
\end{equation*}
Values for $s$ for which $\Delta_{6}(s)$ differs from $\Delta'_{6}(s)$  by a square are in one-to-one correspondence with the rational points on the elliptic curve
\begin{equation*}
O_{6}:y^2=x^3 + 17x^2 + 72x.
\end{equation*}
$O_{6}(\mathbb{Q})$ contains only three rational points, namely $(0,0),(-8,0),(-9,0)$.\\

\textit{Degree 8:}
For the pair of geometric isomorphism classes
\begin{equation*}
j_{8}(s)=\dfrac{(s^4+16s^3+80s^2+128s+16)^3}{s(s+4)^2(s+8)}, \ j'_{8}(s)=\dfrac{(s^4+256s^3+5120s^2+32768s+65536)^3}{s^8(s+4)(s+8)^2},
\end{equation*}
a pair elliptic curves, up to a square, have discriminant:
\begin{equation*}
\Delta_{8}(s)=s(s+8), \ \Delta'_{8}(s)=(s+4).
\end{equation*}

Values for $s$ for which $\Delta_{8}(s)$ differs from $\Delta'_{8}(s)$  by a square are in one-to-one correspondence with the rational points on the elliptic curve
\begin{equation*}
O_{8}:y^2=x^3 + 12x^2 + 32x.
\end{equation*}
$O_{8}(\mathbb{Q})$ contains only three rational points, namely $(0,0),(-4,0),(-8,0)$. \\

\textit{Degree 10:}
For the pair of geometric isomorphism classes
\begin{align*}
j_{10}(s)&=\dfrac{(s^6+20s^5+160s^4+640s^3+1280s^2+1040s+80)^3}{s(s+4)^5(s+5)^2},\\
j'_{10}(s)&=\dfrac{(s^6+260 s^5+6400 s^4+64000 s^3+320000 s^2+800000 s+800000)^3}{ s^{10} (s+4)^2 (s+5)^5},
\end{align*}
a pair elliptic curves, up to a square, have discriminant:
\begin{equation*}
\Delta_{10}(s)=\dfrac{s(s+4)}{s^2 + 8s + 20}, \  \Delta'_{10}(s)=\dfrac{(s+5)}{s^2 + 8 s + 20}.
\end{equation*}
Values for $s$ for which $\Delta_{10}(s)$ differs from $\Delta'_{10}(s)$  by a square are in one-to-one correspondence with the rational points on the elliptic curve
\begin{equation*}
O_{10}:y^2=x^3 + 9x^2 + 20x.
\end{equation*}
$O_{10}(\mathbb{Q})$ contains only three rational points, namely $(0,0),(-4,0),(-5,0)$. \\

\textit{Degree 12:}
For the pair of geometric isomorphism classes

\begin{align*}
j_{12}(s)&=\dfrac{(s^2+6 s+6)^3 (s^6+18 s^5+126 s^4+432 s^3+732 s^2+504 s+24)^3}{s (s+2)^3 (s+3)^4 (s+4)^3 (s+6)},\\
j'_{12}(s)&=\dfrac{(s^2+12 s+24)^3 (s^6+252 s^5+4392 s^4+31104 s^3+108864 s^2+186624 s+124416)^3}{ s^{12} (s+2) (s+3)^3 (s+4)^4 (s+6)^3},
\end{align*}
a pair elliptic curves, up to a square, have discriminant:
\begin{equation*}
\Delta_{12}(s)=s(s+2)(s+4)(s+6), \ \Delta'_{12}(s)=(s+2)(s+3)(s+6).
\end{equation*}

Values for $s$ for which $\Delta_{12}(s)$ differs from $\Delta'_{12}(s)$  by a square are in one-to-one correspondence with the rational points on the elliptic curve
\begin{equation*}
O_{12}:y^2=x^3 + 7x^2 + 12x.
\end{equation*}
$O_{12}(\mathbb{Q})$ contains only three rational points, namely $(0,0),(-3,0),(-4,0)$.\\

\textit{Degree 16:}
For the pair of geometric isomorphism classes
\begin{align*}
&j_{16}(s)=\dfrac{(s^8+16 s^7+112 s^6+448 s^5+1104 s^4+1664 s^3+1408 s^2+512 s+16)^3}{s (s+2)^4 (s+4) (s^2+4 s+8)},\\
&j'_{16}(s)=\frac{(s^8+256 s^7+5632 s^6+53248 s^5+282624 s^4+917504 s^3+1835008 s^2+2097152 s+1048576)^3}{s^{16} (s+2) (s+4)^4 (s^2+4 s+8)},
\end{align*}
a pair elliptic curves, up to a square, have discriminant:
\begin{equation*}
\Delta_{16}(s)=s(s+4)(s^2 + 4 s + 8), \ \Delta'_{16}(s)=(s+2)(s^2 + 4 s + 8).
\end{equation*}

Values for $s$ for which $\Delta_{16}(s)$ differs from $\Delta'_{16}(s)$  by a square are in one-to-one correspondence with the rational points on the elliptic curve
\begin{equation*}
O_{16}:y^2=x^3 + 6x^2 + 8x.
\end{equation*}
$O_{16}(\mathbb{Q})$ contains only three rational points, namely $(0,0),(-2,0),(-4,0)$. \\

\textit{Degree 18:}
For the pair of geometric isomorphism classes
\begin{align*}
j_{18}(s)=&(s^3+6 s^2+12 s+6)^3 \\
&\cdot(s^9+18 s^8+144 s^7+666 s^6+1944 s^5+3672 s^4+4404 s^3+3096 s^2+1008 s+24)^3\\\
 &\cdot s^{-1} (s+2)^{-9} (s+3)^{-2} (s^2+3 s+3)^{-2} (s^2+6 s+12)^{-1}, \\
 j'_{18}(s)=&(s^3+12 s^2+36 s+36)^3 \\
& \cdot(s^9+252 s^8+4644 s^7+39636 s^6+198288 s^5+629856 s^4+1294704 s^3
+1679616 s^2\\
&+1259712 s+419904)^3 \cdot s^{-18} (s+2)^{-2} (s+3)^{-9} (s^2+3 s+3)^{-1} (s^2+6 s+12)^{-2}.
\end{align*}
a pair elliptic curves, up to a square, have discriminant:
\begin{equation*}
\Delta_{18}(s)=s(s+2)(s^2 + 6 s + 12), \  \Delta'_{18}(s)= (s+3)(s^2 + 3 s + 3).
 \end{equation*}

Values for $s$ for which $\Delta_{18}(s)$ differs from $\Delta'_{18}(s)$  by a square are in one-to-one correspondence with the rational points on the genus $3$ hyperelliptic curve
\begin{equation*}
O_{18}:y^2=x(x+2)(x+3)(x^2+3x+3)(x^2+6x+12).
\end{equation*}
But, $O_{18}(\mathbb{Q})$ contains only finitely many points by Faltings theorem, see \cite{Falt}.

\bibliographystyle{alpha}
\bibliography{bib}
 
\end{document}